\pdfoutput=1
\RequirePackage{ifpdf}
\ifpdf 
\documentclass[pdftex]{sigma}
\else
\documentclass{sigma}
\fi

\usepackage[all]{xy}

\numberwithin{equation}{section}

\newtheorem{Theorem}{Theorem}[section]

\newtheorem{Proposition}[Theorem]{Proposition}
 { \theoremstyle{definition}
\newtheorem{Definition}[Theorem]{Definition}

\newtheorem{Remark}[Theorem]{Remark} }

\begin{document}

\allowdisplaybreaks

\newcommand{\arXivNumber}{1710.03977}

\renewcommand{\PaperNumber}{111}

\FirstPageHeading

\ShortArticleName{Moduli of Parabolic Connetions with Quadratic Differential}

\ArticleName{The Moduli Spaces of Parabolic Connections\\ with a Quadratic Differential\\
and Isomonodromic Deformations}

\Author{Arata KOMYO}

\AuthorNameForHeading{A.~Komyo}

\Address{Department of Mathematics, Graduate School of Science, Osaka University, \\
Toyonaka, Osaka 560-0043, Japan}
\Email{\href{mailto:a-koumyou@cr.math.sci.osaka-u.ac.jp}{a-koumyou@cr.math.sci.osaka-u.ac.jp}}

\ArticleDates{Received January 23, 2018, in final form October 03, 2018; Published online October 13, 2018}

\Abstract{In this paper, we study the moduli spaces of parabolic connections with a~quad\-ratic differential. We endow these moduli spaces with symplectic structures by using the fundamental 2-forms on the moduli spaces of parabolic connections (which are phase spaces of isomonodromic deformation systems). Moreover, we see that the moduli spaces of parabolic connections with a quadratic differential are equipped with structures of twisted cotangent bundles.}


\Classification{14D20; 34M56}

\section{Introduction}

\looseness=1 Let $C$ be a smooth projective curve of genus $g$ (where $g\ge 2$). Narasimhan--Seshadri \cite{NS} showed that vector bundles on $C$ are stable if and only if they arise from irreducible unitary representations of the fundamental group of $C$. The moduli space of stable vector bundles on $C$ is equipped with a natural symplectic structure. Although the complex structure on this moduli space depends on the complex structure of $C$, the symplectic structure of this moduli space depends only on the underlying topological surface of $C$. This picture has been investigated by Atiyah--Bott \cite{AB} and Goldman \cite{Goldman}. There exist generalizations of this picture. One can consider the moduli space of pairs $(E,\nabla)$ where $(E,\nabla)$ is a rank $r$ vector bundle on $C$ with a holomorphic connection $\nabla$. This moduli space is equipped with a (holomorphic) symplectic structure. There exists an analytic isomorphism between the moduli space of pairs $(E,\nabla)$ and the moduli space of representations of the fundamental group of $C$ into $\mathrm{GL}(r,\mathbb{C})$ by taking a holomorphic connection to its monodromy representation. Considering the variation of this isomorphism when deforming the curve, we can define the isomonodromic foliation on the moduli space of triples $(C, E,\nabla)$. This foliation is transversal to the fibration $(C,E,\nabla)\mapsto C$ of complementary dimension. There exists a closed $2$-form on the moduli space of triples $(C,E,\nabla)$ such that the kernel of the closed $2$-form coincides with the tangent spaces of leaves of the foliation and this $2$-form induces a (holomorphic) symplectic structure on the moduli space of pairs $(E,\nabla)$ over a fixed curve $C$. This generalization has been investigated by Goldman \cite{Goldman}, Hitchin \cite{Hit1}, and Simpson \cite{Simp1,Simp2}. Moreover, this generalized picture was generalized to the singular setting by Iwasaki \cite{Iwa1}, Hitchin \cite{Hit2}, Boalch \cite{Boalch}, and Krichever \cite{Krich}. Remark that, in the logarithmic case, Inaba--Iwasaki--Saito \cite{IIS} and Inaba \cite{Inaba} have constructed the moduli \textit{scheme} of triples $(C,E,\nabla)$ (satisfying some stability condition) and showed that the closed 2-form on this moduli scheme is algebraic.

We recall the definitions of Lagrangian triples and Hamiltonian data, which are discussed in~\cite{BK}. Let $p\colon X \rightarrow S$ be a smooth morphism of smooth varieties. \textit{A $p$-connection} is an $\mathcal{O}_X$-linear morphism $\nabla_S\colon p^* \Theta_S \rightarrow \Theta_X$ such that $d p \circ \nabla_S = \mathrm{id}_{p^*\Theta_S}$. Here $\Theta_S$ and $\Theta_X$ are the tangent sheaves of $S$ and $X$, respectively. A $p$-connection $\nabla_S$ is \textit{integrable} if the corresponding map $\Theta_S \rightarrow p_* \Theta_X$ commutes with brackets. Note that an integrable $p$-connection $\nabla_S$ defines an action of $\Theta_S$ on relative differential forms $\Omega_{X/S}$ by the Lie derivatives along horizontal vector field $\nabla_{S}(\Theta_S)$. A form $\omega\in \Omega^2_{X/S}$ is \textit{$\nabla_S$-horizontal} if $\omega$ is fixed by the $\Theta_S$-action.

\begin{Definition}\looseness=-1 Let $X$ be a smooth algebraic variety over $\mathbb{C}$ and $T^*=T^*(X) \rightarrow X $ be the cotangent bundle on $X$. A {\it twisted cotangent bundle} on $X$ is a $T^*$-torsor $\pi_{\phi} \colon \phi \rightarrow X$ (i.e., $\pi_{\phi}$ is a fibration equipped with a~simple transitive action of $T^*$ along the fibers) together with a~symplectic form $\omega_{\phi}$ on $\phi$ such that $\pi_{\phi}$ is a polarization for $\omega_{\phi}$ (i.e., $ \dim \phi =2 \dim X $ and the Poisson bracket $\{ \cdot , \cdot \}$ vanishes on $\pi_{\phi}^* \mathcal{O}_X$) and for any $1$-form $\nu$ on an open set $ U \subset X$ one has
$t^*_{\nu} (\omega_{\phi}) = \pi_{\phi}^* d \nu + \omega_{\phi}$ on $\pi_{\phi}^{-1}(U)$. Here $t_{\nu} \colon \pi_{\phi}^{-1}(U) \rightarrow \pi_{\phi}^{-1}(U)$; $t_{\nu}(a) = a + \nu_{\pi(a)}$ is the translation by~$\nu$.
\end{Definition}

For example, the map from the moduli space of pairs $(E,\nabla)$ to the moduli space of vector bundles defined by $(E,\nabla) \mapsto E$ is a twisted cotangent bundle on the moduli space of vector bundles (see \cite[Lemma~IV.4]{Faltings} and \cite[Section~4]{BF}). Note that this moduli space of vector bundles is a smooth algebraic stack. We can define a twisted cotangent bundle on a smooth algebraic stack in the same way.

\begin{Definition}Let $S$ be a smooth variety. An \textit{$S$-Lagrangian triple} consists of a morphism $\pi\colon X \rightarrow Y$ of $S$-varieties $p_X \colon X\rightarrow S$ and $p_Y\colon Y \rightarrow S$, a relative 2-form $\omega \in \Omega_{X/S}^2(X)$ and a~$p_X$-connection $\nabla_S$ such that
\begin{itemize}\itemsep=0pt
\item[(i)] $p_X$, $p_Y$ and $\pi$ are smooth surjective morphisms,
\item[(ii)] the form $\omega$ is closed and non-degenerate,
\item[(iii)] for any $s \in S$ the morphism $\pi_s \colon X_s \rightarrow Y_s$ is a twisted cotangent bundle over $Y_s$, and
\item[(iv)] $\nabla_S$ is integrable and $\omega$ is $\nabla_S$-horizontal.
\end{itemize}
\end{Definition}

\begin{Definition}\label{Hamiltonian datum} An \textit{$S$-Hamiltonian datum} on an $S$-variety $p_Y \colon Y\rightarrow S$ consists of
\begin{itemize}\itemsep=0pt
\item[(i)] a twisted cotangent bundle $\big(\widetilde{X}, \omega_{\widetilde{X}}\big)$, $\tilde{\pi}\colon \widetilde{X} \rightarrow Y$ over $Y$. Put $X := \widetilde{X} \ \mathrm{mod}\ p^*_Y \Omega_S^1$: this is a $\Theta_{Y/S}^*$-torsor over $Y$; let $\widetilde{X} \xrightarrow{r} X \xrightarrow{\pi} Y$ be the projections and
\item[(ii)] a section $h \colon X \rightarrow \widetilde{X}$ of $r$ (called \textit{Hamiltonian})
\end{itemize}
such that for each $x\in X$ the form $(\omega_X)_x \in \bigwedge^2 \Theta^*_{X,x}$ has rank $\dim X -\dim S$. Here we put $\omega_X := h^* \omega_{\widetilde{X}}$, which is a closed 2-form on $X$.
\end{Definition}

Remark that the twisted cotangent bundle $\widetilde{X}$ over $Y$ is isomorphic to the fiber product \smash{$X\times_S T^*S$} as symplectic manifolds. This isomorphism is given by the morphism $\tilde{r} \colon \widetilde{X} \rightarrow$ \smash{$X\times_S T^*S$}, $\tilde{r}(\tilde{x}) = (r(\tilde{x}), \tilde{x} - h(r(\tilde{x})))$. Here the symplectic form on $X\times_S T^*S$ is equal to the sum of~$\omega_X$ and a standard symplectic form on $T^*S$. Now, we describe a construction of $S$-Lagrangian triples from $S$-Hamiltonian data $\big(\widetilde{X}, \omega_{\widetilde{X}}, \tilde{\pi}, h \big)$. Let $\pi \colon X\rightarrow Y$ be the map as in Definition~\ref{Hamiltonian datum}. For each $s \in S$, the map $\pi_s \colon X_s \rightarrow Y_s$ is a $T^*(Y_s)$-torsor induced by the $\Theta_{Y/S}^*$-torsor $X \rightarrow Y$. Let~$\omega$ be the image of $\omega_{X}$ under the natural morphism $\Omega^2_X(X)\rightarrow \Omega^2_{X/S}(X)$. For the natural map $\iota_{X_s} \colon X_s \rightarrow X$, the pull-back $\iota^*_{X_s} \omega$ is a symplectic form on $X_s$. Let $a \in \Omega^1_{Y_s}(U)$ be a~local section over an open set $U \subset Y_s$. We take a collection $\{ (U_i, \tilde{a}_i) \}_i$ where $\{ U_i\}_i$ is an open covering of $U$ and $\tilde{a}_i \in \Omega^1_{Y}|_{Y_s}(U_i)$ such that $\iota_{Y_s}^*(\tilde{a}_i) = a|_{U_i}$, where $\iota_{Y_s} \colon Y_s \rightarrow Y$ is the natural map. We can show that $t_{a}^*( \iota^*_{X_s} h^* \omega_{\widetilde{X}}) = (h|_{X_s})^* t_{\tilde{a}_i}^* (\iota^*_{\widetilde{X}_s} \omega_{\widetilde{X}})$ on $\pi^{-1}_s(U_i)$, where $\iota_{\widetilde{X}_s} \colon \widetilde{X}_s \rightarrow \widetilde{X}$ is the natural map. (Here note that $\widetilde{X}$ is isomorphic to $X\times_S T^*S$.) In particular, the right-hand side is independent of the choice of a lift $\tilde{a}_i$ of $a|_{U_i}$. Then $t_{a}^*( \iota^*_{X_s} \omega) - \iota^*_{X_s} \omega = (h|_{X_s})^* t_{\tilde{a}_i}^* (\iota^*_{\widetilde{X}_s} \omega_{\widetilde{X}}) - (h|_{X_s})^*(\iota^*_{\widetilde{X}_s} \omega_{\widetilde{X}}) =(h|_{X_s})^* (\tilde{\pi}|_{\widetilde{X}_s})^* d ( \tilde{a}_i) =\pi_s^*da$. We have that $\pi_s \colon X_s \rightarrow Y_s$ is a twisted cotangent bundle. Put $p_X := p_Y \circ \pi$. The kernels of~$(\omega_X)_x$ for each $x \in X$ form a subbundle of the tangent bundle $TX$, which is transversal to fibers of $p_X$. Since the form $(\omega_X)_x \in \bigwedge^2 \Theta^*_{X,x}$ has rank $\dim X -\dim S$ and $\omega_X$ is closed, this subbundle defines an integrable $p_X$-connection $\nabla_{S}$. By the construction, $\omega$ is $\nabla_S$-horizontal. Then $(\pi \colon X\rightarrow Y,\omega ,\nabla_S )$ is an $S$-Lagrangian triple. The purpose of this paper is to construct $S$-Hamiltonian data $\big(\widetilde{X}, \omega_{\widetilde{X}}, \tilde{\pi}, h \big)$ from $S$-Lagrangian triples $(\pi \colon X\rightarrow Y,\omega_X ,\nabla_S )$ \textit{by using concrete argument} in the case of isomonodromic deformations. (There exists a more abstract construction in~\cite{BK} for a general case.) Now, following~\cite{BK}, we describe that the Hamiltonian $h \colon X \rightarrow \widetilde{X}$ of an $S$-Hamiltonian datum is locally given by local functions and the integrable $p_X$-connection $\nabla_S$ associated to the $S$-Hamiltonian datum has a~description by these functions. Let $x$ be a~point of~$X$. Let $y= \pi (x)$ and $s= p_X (x)$ be the projections of~$x $. Let $( t_a )_{a=1,\ldots,\dim S}$ be local coordinates on a neighborhood of $s \in S$ and~$q_i$, $i=1,\ldots, \dim Y_s$, be functions on a neighborhood of $y\in Y$ such that $( q_i, t_a )_{i,a}$ are local coordinates at $y$ on $Y$. Here we denote pull-backs of local functions by the same notations as the local functions for simplicity. Choose functions~$h_a$ and~$p_i$ on a neighborhood of $h(x) \in \widetilde{X}$ such that
\begin{gather*} \omega_{\widetilde{X}} = \sum\limits_{i=1}^{\dim Y_s} d p_i \wedge d q_i + \sum\limits_{a=1}^{\dim S} d h_a \wedge d t_a .\end{gather*} Then $( \boldsymbol{q}= ( q_i )_i$, $\boldsymbol{p}= ( p_i )_i$, $\boldsymbol{t}= ( q_a )_a )$ are local coordinates at~$x$ on~$X$. The Hamiltonian $h \colon X \rightarrow \widetilde{X}$ is given by the functions $h_a(\boldsymbol{q},\boldsymbol{p},\boldsymbol{t} )$. Note that
\begin{gather*}
\omega_{X} = \sum\limits_{i=1}^{\dim Y_s} d p_i \wedge d q_i + \sum\limits_{a=1}^{\dim S} d h_a(\boldsymbol{q},\boldsymbol{p},\boldsymbol{t} ) \wedge d t_a .
\end{gather*} Put
\begin{gather*} v_{h_a}= \partial_{t_a} + \sum\limits_{i=1}^{\dim Y_s} \partial_{q_i}(h_a(\boldsymbol{q},\boldsymbol{p},\boldsymbol{t} )) \partial_{p_i} -\partial_{p_i}(h_a(\boldsymbol{q},\boldsymbol{p},\boldsymbol{t} ))\partial_{q_i}. \end{gather*} We can check $\omega_X(\partial_{p_i} ,v_{h_a} )=\omega_X(\partial_{q_i} ,v_{h_a} )=0$
easily. Moreover we have $\omega_X(\partial_{t_b} ,v_{h_a} )=0$, $a,b=1,\ldots,\dim S$, since for each $x\in X$ the form $(\omega_X)_x \in \bigwedge^2 \Theta^*_{X,x}$ has rank $\dim X -\dim S$. Then we have a description of $\nabla_S$ by $h_a(\boldsymbol{q},\boldsymbol{p},\boldsymbol{t} )$: $\nabla_S(\partial_{t_a})= v_{h_a}$.

In this paper, we consider an $S$-Lagrangian triples $(\pi \colon X\rightarrow Y,\omega_X ,\nabla_S )$ associated to isomonodromic deformations of parabolic connections (which are logarithmic connection with quasi-parabolic structures). The isomonodromic deformations of parabolic connections have been investigated by Inaba--Iwasaki--Saito~\cite{IIS} and Inaba~\cite{Inaba}. In our case, $X$ is a moduli space of pointed smooth projective curves and parabolic connections (see \cite[Theorem~2.1]{Inaba} and~\cite{IIS}), $Y$~is a moduli space of pointed smooth projective curves and quasi-parabolic bundles admitting a parabolic connection, and $S$ is a moduli space of pointed smooth projective curves. We have projections $p_X\colon X \rightarrow S$, $p_Y\colon Y \rightarrow S$ and $\pi \colon X \rightarrow Y$. The moduli space $X$ has the relative symplectic form $\omega$ over $S$ (see \cite[Section~7]{Inaba}). The $p_X$-connection $\nabla_S$ is given by the isomonodromic deformations of parabolic connections (see \cite[Proposition~8.1]{Inaba}). The main result of this paper is to construct the corresponding twisted cotangent bundle $\widetilde{X}$ over $Y$ \textit{by using computation of \v{C}ech cohomologies}. We construct the twisted cotangent bundle with the remark in mind: The twisted cotangent bundle $\widetilde{X}$ over $Y$ is isomorphic to the fiber product $X\times_S T^*S$.

Our argument is as follows. First, we consider the fiber product $X\times_S T^*S$ (which called \textit{extended phase space}, see \cite[Section~7]{Hurt}). The fiber product $X\times_S T^*S$ is \textit{the moduli space of} (\textit{pointed smooth projective curves and}) \textit{parabolic connections with a quadratic differential}. We describe the tangent sheaf of $X\times_S T^*S$ and the symplectic form on $X\times_S T^*S$ by the \v{C}ech cohomology (Propositions~\ref{Tangent sheaf Cech 1} and~\ref{Prop symplectic}). Second, we describe the cotangent sheaf $\Omega_{Y}^1$ by the \v{C}ech cohomology, and we define an $\Omega_{Y}^1$-action on $X\times_S T^*S$ explicitly (Definition~\ref{Omega action}). We show that by this $\Omega_{Y}^1$-action and the symplectic form, $X\times_S T^*S$ is a twisted cotangent bundle over~$Y$ (Theorem~\ref{Main Thm 2}). The section $X \rightarrow X\times_S T^*S$ given by the zero section of $T^*S \rightarrow S$ is the Hamiltonian of the Hamiltonian datum.

A twisted cotangent bundle over $Y$ is important for studying quantizations of isomonodromic deformations. In fact, quantizations of isomonodromic deformations may be described by using certain algebras of twisted differential operators, which are quantizations of twisted cotangent bundles (see \cite{BK, BF}). It is expected that the results of this paper are useful to understand quantizations of isomonodromic deformations in the context of a certain algebro-geometric way such as \cite{Inaba,IIS}.

The organization of this paper is as follows. In Section \ref{Pre}, we recall basic definitions and basic facts on parabolic connections (in Section~\ref{SS ParaConn}), Atiyah algebras (in Section~\ref{SS AA}) and twisted cotangent bundles (in Section~\ref{TDO on SAV}). In Section~\ref{Moduli of ParaConn with QuadDiff}, we treat moduli spaces of parabolic connections with a~quadratic differential. First, we describe the tangent sheaves of these moduli spaces in terms of the hypercohomology of a certain complex. Second, we endow the moduli spaces with symplectic structures. In Section~\ref{Section Twisted Cotangent Bundle}, we see that the moduli spaces of parabolic connections with a quadratic differential are equipped with structures of twisted cotangent bundles.

\section{Preliminaries}\label{Pre}

\subsection{Moduli space of stable parabolic connections}\label{SS ParaConn}
Following \cite{Inaba}, we recall basic definitions and basic facts on parabolic connections. Let $C$ be a~smooth projective curve of genus $g$. We put
\begin{gather*}
T_n := \{ (t_1, \ldots,t_n ) \in C \times \cdots \times C \,|\, t_i \neq t_j \ \text{for $i\neq j$} \}
\end{gather*}
for a positive integer $n$. For integers $e$, $r$ with $r>0$, we put
\begin{gather*}
N^{(n)}_r(e) := \bigg\{ \big(\nu^{(i)}_j\big)^{1\le i \le n}_{0\le j \le r-1} \in \mathbb{C}^{nr} \, \bigg| \, e+ \sum_{i,j} \nu^{(i)}_j =0 \bigg\}.
\end{gather*}
Take members $\boldsymbol{t}= (t_1,\ldots,t_n) \in T_n$ and $\boldsymbol{\nu}=\big( \nu^{(i)}_j\big)_{1\le i \le n,0\le j \le r-1} \in N^{(n)}_r(e)$.

\begin{Definition} We say $\big(E,\nabla, \big\{ l_*^{(i)} \big\}_{1\le i\le n}\big)$ is a \textit{$(\boldsymbol{t} , \boldsymbol{\nu})$-parabolic connection of rank~$r$ and degree~$e$ over~$C$} if
\begin{itemize}\itemsep=0pt
\item[(1)] $E$ is a rank $r$ algebraic vector bundle on $C$,
\item[(2)] $\nabla \colon E \rightarrow E \otimes \Omega^1_C(t_1 + \cdots + t_n)$ is a connection, that is, $\nabla$ is a $\mathbb{C}$-linear homomorphism of sheaves satisfying $\nabla(fa)=a \otimes df + f \nabla(a)$ for $f \in \mathcal{O}_C$ and $a \in E$, and
\item[(3)] for each $t_i$, $l_*^{(i)}$ is a filtration $E|_{t_i} = l_0^{(i)} \supset l_1^{(i)} \supset \cdots \supset l_r^{(i)}=0$ such that $\dim \big(l_j^{(i)}/l_{j+1}^{(i)}\big)=1$ and $\big({\sf res}_{t_i}(\nabla)-\nu^{(i)}_j \mathrm{id}_{E|_{t_i}}\big) \big(l^{(i)}_j\big) \subset l_{j+1}^{(i)}$ for $j=0,\ldots,r-1$.
\end{itemize}
\end{Definition}

\begin{Remark}We have
\begin{gather*}
\deg E = \deg (\det(E)) = -\sum^n_{i=1} \mathrm{tr} ({\sf res}_{t_i} (\nabla)) = - \sum^n_{i=1} \sum^{r-1}_{j=0} \nu_j^{(i)} =e.
\end{gather*}
\end{Remark}

\begin{Definition}[{\cite[Definition 2.3]{Inaba}}] Take an element $\boldsymbol{\nu} \in N^{(n)}_r(e)$. We call $\boldsymbol{\nu}$ \textit{special} if
\begin{itemize}\itemsep=0pt
\item[(1)] $\nu_j^{(i)}-\nu_k^{(i)} \in \mathbb{Z}$ for some $i$ and $j\neq k$, or
\item[(2)] there exists an integer $s$ with $1<s<r$ and a subset $\{ j_1^i,\ldots,j_s^i \} \subset \{ 0,\ldots,r-1 \}$ for each $1\le i \le n$ such that $\sum\limits_{i=1}^{n} \sum\limits^{s}_{k=1} \nu_{j^i_k}^{(i)} \in \mathbb{Z}$.
\end{itemize}
We call $\boldsymbol{\nu}$ \textit{generic} if it is not special.
\end{Definition}

Take rational numbers $0 < \alpha^{(i)}_1 < \alpha^{(i)}_2 < \cdots < \alpha^{(i)}_r <1$ for $i=1,\ldots , n$ satisfying $\alpha^{(i)}_j \neq \alpha^{(i')}_{j'}$ for $(i,j) \neq (i', j')$. We choose a sufficiently generic $\boldsymbol{\alpha}=(\alpha_j^{(i)})$.

\begin{Definition}A parabolic connection $\big(E,\nabla, \big\{ l^{(i)}_* \big\}_{1\le i \le n}\big)$ is \textit{$\boldsymbol{\alpha}$-stable} (resp.\ \textit{$\boldsymbol{\alpha}$-semistable}) if for any proper nonzero subbundle $F\subset E$ satisfying $\nabla(F) \subset F \otimes \Omega^1_C(t_1 + \cdots +t_n)$, the inequality
\begin{gather*}
\frac{\deg F+\sum\limits^n_{i=1}\sum\limits^r_{j=1} \alpha^{(i)}_j \dim \big( \big(F|_{t_i} \cap l_{j-1}^{(i)}\big)/\big(F|_{t_i}\cap l_j^{(i)}\big) \big) }{\operatorname{rank} F} \\
\qquad{} \underset{(\text{resp. $\le$})}{<} \frac{\deg E+\sum\limits^n_{i=1}\sum\limits^r_{j=1}\alpha^{(i)}_j \dim \big( l_{j-1}^{(i)}/ l_j^{(i)} \big)}{ \operatorname{rank} E}
\end{gather*}
holds.
\end{Definition}

Let $\tilde{M}_{g,n}$ be a smooth algebraic scheme which is a certain covering of the moduli stack of $n$-pointed smooth projective curves of genus $g$ over $\mathbb{C}$ and take a universal family $\big(\mathcal{C},\tilde{t}_1, \ldots,\tilde{t}_n\big)$ over $\tilde{M}_{g,n}$.

\begin{Definition}We denote the pull-back of $\mathcal{C}$ and $\tilde{\boldsymbol{t}}$ by the morphism $\tilde{M}_{g,n} \times N^{(n)}_r(e) \rightarrow \tilde{M}_{g,n}$ by the same character $\mathcal{C}$ and $\tilde{\boldsymbol{t}}=\big\{ \tilde{t}_1 ,\ldots,\tilde{t}_n \big\}$. Then $D\big(\tilde{\boldsymbol{t}}\big):=\tilde{t}_1 +\cdots+\tilde{t}_n$ becomes an effective Cartier divisor on $\mathcal{C}$ flat over $\tilde{M}_{g,n} \times N^{(n)}_r(e)$. We also denote by $\tilde{\boldsymbol{\nu}}$ the pull-back of the universal family on $N^{(n)}_r(e)$ by the morphism $\tilde{M}_{g,n} \times N^{(n)}_r(e) \rightarrow N^{(n)}_r(e)$. We define a functor $\mathcal{M}_{\mathcal{C}/\tilde{M}_{g,n}}^{\boldsymbol{\alpha}}\big(\tilde{\boldsymbol{t}},r,e\big)$ from the category of locally noetherian schemes over $\tilde{M}_{g,n} \times N^{(n)}_r(e)$ to the category of sets by
\begin{gather*}
\mathcal{M}_{\mathcal{C}/\tilde{M}_{g,n}}^{\boldsymbol{\alpha}}\big(\tilde{\boldsymbol{t}},r,e\big)(S):=\big\{ \big(E,\nabla, \big\{l^{(i)}_j \big\}\big) \big\}/{\sim}
\end{gather*}
for a locally noetherian scheme $S$ over $\tilde{M}_{g,n} \times N^{(n)}_r(e)$, where
\begin{itemize}\itemsep=0pt
\item[(1)] $E$ is a rank $r$ algebraic vector bundle on $\mathcal{C}_S$,
\item[(2)] $\nabla \colon E \rightarrow E \otimes \Omega^1_{\mathcal{C}_S/S}\big(D\big(\tilde{\boldsymbol{t}}\big)_S\big)$ is a relative connection,
\item[(3)] for each $(\tilde{t}_i)_S$, $l_*^{(i)}$ is a filtration by subbundles $E|_{(\tilde{t}_i)_S} = l_0^{(i)} \supset l_1^{(i)} \supset \cdots \supset l_r^{(i)}=0$ such that $\big({\sf res}_{(\tilde{t}_i)_S}(\nabla)-\big(\tilde{\nu}^{(i)}_j\big)_S \mathrm{id}_{E|_{t_i}}\big) \big(l^{(i)}_j\big) \subset l_{j+1}^{(i)}$ for $j=0,\ldots,r-1$, and
\item[(4)] for any geometric point $s \in S$, $\dim \big(l_j^{(i)}/l_{j+1}^{(i)}\big)\otimes k(s)=1$ for any $i$, $j$ and $\big(E,\nabla,\big\{ l^{(i)}_j \big\}\big) \otimes k(s)$ is $\boldsymbol{\alpha}$-stable.
\end{itemize}
Here $\big(E, \nabla,\big\{ l^{(i)}_j \big\}\big)\sim \big(E', \nabla',\big\{ l'^{(i)}_j \big\}\big)$ if there exist a line bundle $\mathcal{L}$ on $S$ and an isomorphism $\sigma \colon E \xrightarrow{\sim} E'\otimes \mathcal{L}$ such that $\sigma|_{t_i}\big(l^{(i)}_J\big) = l'^{(i)}_j \otimes \mathcal{L}$ for any $i$, $j$ and the diagram{\samepage
\begin{gather*}
\xymatrix{
E \ar[r]^-{\nabla} \ar[d]^-{\sigma} & E \otimes \Omega^1_{\mathcal{C}/T}\big(D\big(\tilde{\boldsymbol{t}}\big)\big)
\ar[d]^-{\sigma\otimes \mathrm{id}} \\
E' \otimes \mathcal{L} \ar[r]^-{\nabla'\otimes \mathcal{L}} & E' \otimes \Omega^1_{\mathcal{C}/T}\big(D\big(\tilde{\boldsymbol{t}}\big)\big)\otimes \mathcal{L}}
\end{gather*}
commutes.}
\end{Definition}

\begin{Theorem}[{\cite[Theorem 2.1]{Inaba}}]For the moduli functor $\mathcal{M}_{\mathcal{C}/\tilde{M}_{g,n}}^{\boldsymbol{\alpha}}\big(\tilde{\boldsymbol{t}},r,e\big)$,
there exists a fine moduli scheme
\begin{gather*}
M_{\mathcal{C}/\tilde{M}_{g,n}}^{\boldsymbol{\alpha}}\big(\tilde{\boldsymbol{t}},r,e\big) \longrightarrow \tilde{M}_{g,n} \times N^{(n)}_r(e)
\end{gather*}
of $\boldsymbol{\alpha}$-stable parabolic connections of rank~$r$ and degree~$e$, which is smooth and quasi-projective. The fiber $M_{C_x}^{\boldsymbol{\alpha}}\big(\tilde{\boldsymbol{t}}_x, \boldsymbol{\nu}\big)$ over $(x,\boldsymbol{\nu}) \in \tilde{M}_{g,n} \times N^{(n)}_r(e)$ is the moduli space of $\boldsymbol{\alpha}$-stable $\big(\tilde{\boldsymbol{t}}_x, \boldsymbol{\nu}\big)$-parabolic connections whose dimension is $2r^2(g-1) +nr(r-1)+2$ if it is non-empty.
\end{Theorem}

\subsection{Atiyah algebras}\label{SS AA}

Following \cite[Section~1]{BS}, we recall the \textit{Atiyah algebra}. Let $C$ be a smooth projective curve, and $\Theta_C$ be the tangent sheaf. Let $E$ be a vector bundle of rank $r$ on $C$. Put $\mathcal{D}_E=\mathcal{D}{\rm iff}(E,E)=\bigcup_i \mathcal{D}_i$, $\mathcal{D}_i$ is the sheaf of differential operators of degree $\le i$ on $E$. We have $\mathcal{D}_i/ \mathcal{D}_{i-1}= \mathcal{E}{\rm nd} (E) \otimes S^i(\Theta_C)$ where $S^i(\Theta_C)$ is the $i$-th symmetric product of $\Theta_C$. Let $\mathrm{symb}_1 \colon \mathcal{D}_1 \rightarrow \mathcal{E}{\rm nd}(E) \otimes \Theta_C$ be the natural morphism $\mathcal{D}_1 \rightarrow \mathcal{D}_1/ \mathcal{O}_C = \mathcal{E}{\rm nd}(E) \otimes \Theta_C$.
\begin{Definition}We define the \textit{Atiyah algebra} of $E$ as
\begin{gather*}
\mathcal{A}_E = \{ \partial \in \mathcal{D}_1 \,|\, \mathrm{symb}_1(\partial) \in \mathrm{id}_E\otimes \Theta_C \subset \mathcal{E}{\rm nd}(E) \otimes \Theta_C \}.
\end{gather*}
Here, for $v\in \mathcal{D}_1$, $\mathrm{symb}_1(v)$ is the symbol of the differential operator $v$.
\end{Definition}
We have inclusions $\mathcal{D}_0= \mathcal{E}{\rm nd} (E) \subset \mathcal{A}_E \subset \mathcal{D}_1$ and the short exact sequence
\begin{gather*}\label{ES Atiyah}
0\longrightarrow \mathcal{E}{\rm nd} (E) \longrightarrow \mathcal{A}_E \xrightarrow{\mathrm{symb}_1} \Theta_C \longrightarrow 0.
\end{gather*}
Fix a positive integer $n$. Let $D=t_1+\cdots+t_n$ be an effective divisor of $C$ where $t_1 ,\ldots ,t_n$ are distinct points of $C$. We put $\mathcal{A}_E(D):=\mathrm{symb}_1^{-1} (\Theta_C(-D))$. Then we have the following exact sequence
\begin{gather*}
0\longrightarrow \mathcal{E}{\rm nd} (E) \longrightarrow \mathcal{A}_E(D) \xrightarrow{\mathrm{symb}_1} \Theta_C(-D) \longrightarrow 0.
\end{gather*}

For a connection $\nabla \colon E \rightarrow E \otimes \Omega^1_{C}(D)$, we define a splitting
\begin{gather}\label{splittingnabla}
\iota (\nabla) \colon \ \Theta_C(-D) \longrightarrow \mathcal{A}_E(D)
\end{gather}
as follows. Let $U$ be an affine open subset of $C$ where we have a trivialization $E|_{U} \cong \mathcal{O}_U^{\oplus r}$. We denote by $Af^{-1} df$ a connection matrix of $\nabla$ on $U$ where $f$ is a local defining equation of~$t_i$ and $A \in M_r (\mathcal{O}_U)$. For an element $g\frac{\partial}{\partial f} \in \Theta_C(-D)(U)$, we define the element $\iota (\nabla)\big(g\frac{\partial}{\partial f}\big) := g \big( \frac{\partial}{\partial f} + A f^{-1} \big) \in \mathcal{A}_E(D)(U)$, which gives a map $\iota(\nabla)(U)\colon \Theta_C(-D)(U) \rightarrow \mathcal{A}_E(D)(U)$. By this map, we obtain the splitting~(\ref{splittingnabla}).

\subsection{Twisted cotangent bundles}\label{TDO on SAV}
Following \cite[Section 2]{BB}, we recall the definition of $\Omega^{\ge 1}_X$-torsors and recall the correspondence between twisted cotangent bundles and $\Omega^{\ge 1}_X$-torsors. Let $X$ be a smooth algebraic variety over~$\mathbb{C}$.

\begin{Definition} Let $d\colon A^n \rightarrow A^{n+1}$ be a morphism of sheaves of abelian groups on $X$, considered as length $2$ complex $A^{\bullet}$ supported in degree $n$ and $n+1$. An \textit{$A^{\bullet}$-torsor} is a pair $(\mathcal{F}, c)$, where $\mathcal{F}$ is an $A^n$-torsor and $c\colon \mathcal{F} \rightarrow A^{n+1}$ is a map such that $c(a + \phi) = d(a) + c(\phi)$ for $a \in A^n, \phi \in \mathcal{F}$.
\end{Definition}

Let $\Omega^{\ge 1}_X := \big(\Omega_X^1 \rightarrow \Omega_X^{2{\rm cl}}\big)$ be the truncated de Rham complex, where $\Omega_X^{2{\rm cl}}$ are closed $2$-forms on $X$. For example, let $T^*X\rightarrow X$ be the cotangent bundle of $X$ and $\theta_X$ be the canonical 1-form on $T^*X$. The cotangent sheaf $\Omega_X^1$ of $X$ is an $\Omega_X^1$-torsor, which is trivial. For the~$\Omega_X^1$-torsor~$\Omega_X^1$, we define a map $c\colon \Omega_X^1 \rightarrow \Omega_X^{2{\rm cl}}$ as follows. Let $U$ be a Zariski open set over~$X$. We assigne $\gamma \in \Omega_X^1(U)$ (which is a section $\gamma \colon U \rightarrow T^*X$ of $T^*X\rightarrow X$ on $U$) to $\gamma^* d \theta_X \in \Omega_X^{2{\rm cl}}(U)$. For $\gamma, \gamma' \in \Omega_X^1(U)$, we have $c(\gamma+ \gamma')-c(\gamma) = (\gamma+ \gamma')^*d \theta_X -(\gamma)^*d \theta_X = d(\gamma+ \gamma')- d \gamma =d \gamma'$. Then the pair $(\Omega_X^1, c)$ is an $\Omega^{\ge 1}_X$-torsor.

We recall the correspondence between twisted cotangent bundles and $\Omega^{\ge 1}_X$-torsors. For any morphism $f \colon X \rightarrow Y$ between algebraic varieties, let $\Gamma(f)$ be the sheaf of set on $Y$ where $\Gamma(f)(U)$ is the set of sections of $f$ over $U$ for each open set $U\subset Y$. If $f \colon X \rightarrow Y$ is a~$T^*Y$-torsor, then $\Gamma(f)$ is an $\Omega_Y^1$-torsor. We consider a twisted cotangent bundle $\pi_{\phi} \colon\phi\rightarrow X$. Then $\Gamma(\pi_{\phi})$ is an $\Omega^1_X$-torsor. We define a map $c \colon \Gamma(\pi_{\phi}) \rightarrow \Omega_X^{2{\rm cl}}$ by $c(\gamma):= \gamma^* (\omega_{\phi})$. We have $c(a+\gamma)-c(\gamma)= \gamma^* t_{a}^* (\omega_{\phi}) - \gamma^* (\omega_{\phi})= da$. Then $(\Gamma(\pi_{\phi}),c)$ is an $\Omega^{\ge 1}_X$-torsor. Conversely, for an $\Omega^{\ge 1}_X$-torsor $(\mathcal{F}, c)$, let $\pi_{\phi}\colon \phi \rightarrow X$ be the space of the torsor $\mathcal{F}$. The symplectic form is defined as the unique form such that for a section $\gamma \in \mathcal{F}$ of $\pi_{\phi}$ the corresponding isomorphism $T^*X \xrightarrow{\sim} \phi$; $0 \rightarrow \gamma$, identi\-fies~$\omega_{\phi}$ with $\omega + \pi^* c(\gamma)$. Here $\omega$ is the canonical symplectic form on the cotangent bundle~$T^*X$.

\section[Moduli scheme of parabolic connections with a quadratic differential]{Moduli scheme of parabolic connections\\ with a quadratic differential}\label{Moduli of ParaConn with QuadDiff}

In this section, we study the moduli space of parabolic connections with a quadratic differential, which is generalization of the moduli space of parabolic connections studied by Inaba--Iwasaki--Saito \cite{IIS} and Inaba \cite{Inaba}. In Section~\ref{SS infinit defor}, we describe the (algebraic) tangent sheaf of this moduli space in terms of the hypercohomology of a certain complex by generalization of the description of the tangent sheaf
of the moduli space of parabolic connections in \cite{Inaba,IIS,Komyo}. Moreover, we describe the analytic tangent sheaf in terms of the hypercohomology of a certain analytic complex as in \cite[Section~7]{Inaba}. This description is more simple than the algebraic one. In Section~\ref{SS Isomonod defor}, we recall the description of the vector fields associated to the isomonodromic deformations in terms of the description of the (algebraic) tangent sheaf as in \cite[Section~6]{Hurt} and \cite[Section~3.3]{Komyo}. In Section~\ref{SS Sympl and Hamilt}, we show that the moduli space of parabolic connections with a~quadratic differential is endowed with a symplectic structures. This is the main purpose of this section. In Section~\ref{Extended phase space}, we consider moduli spaces of parabolic connections with a~quadratic differential as extended phase spaces of isomonodromic deformations. The classical trick of turning a time dependent Hamiltonian flow into an autonomous one by adding variables is well-known. In this trick, the space given by adding the variables to a phase space is called an extended phase space. (Hamiltonians of isomonodromic deformations are time dependent.)

Hurtubise \cite{Hurt} also studied the moduli space of connections with a quadratic differential. In \cite[Section~7]{Hurt}, the moduli space of connections with a quadratic differential is decomposed locally into a product of the symplectic manifolds: the moduli space of connections on a fixed curve and the cotangent bundle of the moduli space of curves. This local decomposition is given by using the isomonodromic deformation of the ``background connection'', which is discussed in \cite[Section~6]{Hurt}. Then we can show that the moduli space of connections with a quadratic differential is endowed with a symplectic structure locally. Moreover, in~\cite{Hurt} the section of the map from the moduli space of connections with a quadratic differential to the moduli space of connections is defined by using the Hamiltonians defined in \cite[Section~6]{Hurt}. On the other hand, in our argument we use the ordinary isomonodromic deformation instead of the isomonodromic deformation of the background connection. Then we have an algebraic symplectic structure on the moduli space of (parabolic) connections with a quadratic differential globally. Our corresponding section of the map from the moduli space of connections with a quadratic differential to the moduli space of connections is defined by using the zero section of the map from the moduli space of curves with a quadratic differential to the moduli space of curves.

\subsection{Moduli space of stable parabolic connections with a quadratic differential}
Let $T^*\tilde{M}_{g,n}$ be the total space of the cotangent bundle of $\tilde{M}_{g,n}$. We denote by $\widehat{M}_{\mathcal{C}/\tilde{M}_{g,n}}^{\boldsymbol{\alpha}}\big(\tilde{\boldsymbol{t}},r,e\big)$ the fiber product of $T^*\tilde{M}_{g,n} \times N^{(n)}_r(e)$ and $M_{\mathcal{C}/\tilde{M}_{g,n}}^{\boldsymbol{\alpha}}\big(\tilde{\boldsymbol{t}},r,e\big)$ over $\tilde{M}_{g,n} \times N^{(n)}_r(e)$:
\begin{gather*}
\xymatrix{\widehat{M}_{\mathcal{C}/\tilde{M}_{g,n}}^{\boldsymbol{\alpha}}\big(\tilde{\boldsymbol{t}},r,e\big) \ar[d]_-{\hat{\pi}} \ar[r] &
M_{\mathcal{C}/\tilde{M}_{g,n}}^{\boldsymbol{\alpha}}\big(\tilde{\boldsymbol{t}},r,e\big) \ar[d]^-{\pi} \\
 T^*\tilde{M}_{g,n} \times N^{(n)}_r(e) \ar[r] & \tilde{M}_{g,n} \times N^{(n)}_r(e).}
\end{gather*}
We call the fiber product $\widehat{M}_{\mathcal{C}/\tilde{M}_{g,n}}^{\boldsymbol{\alpha}}\big(\tilde{\boldsymbol{t}},r,e\big)$ the \textit{moduli space of $\boldsymbol{\alpha}$-stable parabolic connections with a quadratic differential}. If we take a zero section of $T^*\tilde{M}_{g,n} \rightarrow \tilde{M}_{g,n}$, then we have an inclusion
\begin{gather*}\label{zerosectionmap}
M_{\mathcal{C}/\tilde{M}_{g,n}}^{\boldsymbol{\alpha}}\big(\tilde{\boldsymbol{t}},r,e\big) \longrightarrow \widehat{M}_{\mathcal{C}/\tilde{M}_{g,n}}^{\boldsymbol{\alpha}}\big(\tilde{\boldsymbol{t}},r,e\big).
\end{gather*}

Let $(C, \boldsymbol{t}) \in \tilde{M}_{g,n}$. The tangent space of $\tilde{M}_{g,n}$ at $(C,\boldsymbol{t})$ is isomorphic to $H^1 (C, \Theta_C(-D(\boldsymbol{t})))$. By the Serre duality, the cotangent space at $(C,\boldsymbol{t})$ is isomorphic to $H^0 \big(C, \Omega^{\otimes 2}_C(D(\boldsymbol{t}))\big)$, which is the space of (global) \textit{quadratic differentials on $(C, \boldsymbol{t})$}.

\subsection{Infinitesimal deformations}\label{SS infinit defor}
For simplicity, we put $\widehat{M}:=\widehat{M}_{\mathcal{C}/\tilde{M}_{g,n}}^{\boldsymbol{\alpha}}\big(\tilde{\boldsymbol{t}},r,e\big)$ and $\mathcal{C}_{\widehat{M}} := \mathcal{C} \times_{\tilde{M}_{g,n}} \widehat{M}$.
Let $\big(\tilde{E}, \tilde{\nabla}, \big\{ \tilde{l}^{(i)}_j \big\}, \tilde{\psi}\big)$ be a~universal family on~$\mathcal{C}_{\widehat{M}}$. Let $\mathcal{A}_{\tilde{E}}\big(D\big(\tilde{\boldsymbol{t}}\big)\big)$ be the relative Atiyah algebra which is the extension
\begin{gather*}
\xymatrix{0\ar[r] & \mathcal{E}{\rm nd}\big(\tilde{E}\big) \ar[r] & \mathcal{A}_{\tilde{E}}\big(D\big(\tilde{\boldsymbol{t}}\big)\big) \ar[r]^-{\mathrm{symb}_1}
& \Theta_{\mathcal{C}_{\widehat{M}}/ \widehat{M}} \big({-}D\big(\tilde{\boldsymbol{t}}\big)\big)
\ar[r] \ar@/^5pt/[l]^-{\iota\big(\tilde{\nabla}\big)} & 0,}
\end{gather*}
where $\iota\big(\tilde{\nabla}\big)\colon \Theta_{\mathcal{C}_{\widehat{M}}/ \widehat{M}} \big({-}D\big(\tilde{\boldsymbol{t}}\big)\big)\rightarrow \mathcal{A}_{\tilde{E}}\big(D\big(\tilde{\boldsymbol{t}}\big)\big)$ is the $\mathcal{O}_X$-linear section of $\mathrm{symb}_1$ associated to the relative connection~$\tilde{\nabla}$. We put
\begin{gather*}
\widetilde{\mathcal{F}}^0 := \big\{ s \in \mathcal{E}{\rm nd}\big(\tilde{E}\big) \,\big|\, s |_{\tilde{t}_i\times \widehat{M}} \big(\tilde{l}^{(i)}_j\big) \subset \tilde{l}^{(i)}_j \text{ for any $i$, $j$} \big\}\qquad \text{and}\\
\mathcal{F}^0 := \big\{ s \in \mathcal{A}_{\tilde{E}}\big(D\big(\tilde{\boldsymbol{t}}\big)\big) \,\big|\, (s - \iota\big(\tilde{\nabla}\big) \circ
 \mathrm{symb}_1(s)) |_{\tilde{t}_i\times \widehat{M}} \big(\tilde{l}^{(i)}_j\big) \subset \tilde{l}^{(i)}_j \text{ for any $i$, $j$} \big\}.
\end{gather*}
Then we have an extension
\begin{gather*}\label{extension F0}
\xymatrix{
0\ar[r] & \widetilde{\mathcal{F}}^0 \ar[r] & \mathcal{F}^0 \ar[r]^-{\mathrm{symb}_1}
& \Theta_{\mathcal{C}_{\widehat{M}}/ \widehat{M}} \big({-}D\big(\tilde{\boldsymbol{t}}\big)\big)
\ar[r] \ar@/^5pt/[l]^-{\iota\big(\tilde{\nabla}\big)} & 0.}
\end{gather*}
We put
\begin{gather*}
\widetilde{\mathcal{F}}^1 := \big\{ s \in \mathcal{E}{\rm nd}\big(\tilde{E}\big) \otimes\Omega^1_{\mathcal{C}_{\widehat{M}}/ \widehat{M}} \big(D\big(\tilde{\boldsymbol{t}}\big)\big) \,\big| \,
{\sf res}_{\tilde{t}_i\times M_{\mathcal{C}/T}^{\boldsymbol{\alpha}}(\tilde{\boldsymbol{t}} ,r,e)} (s) \big(\tilde{l}^{(i)}_j\big)\subset \tilde{l}^{(i)}_{j+1} \text{ for any $i$, $j$}\big\} \qquad \text{and} \\
\mathcal{F}^1 := \widetilde{\mathcal{F}}^1 \oplus \Omega_{\mathcal{C}_{\widehat{M}}/ \widehat{M}}^{\otimes 2} \big(D\big(\tilde{\boldsymbol{t}}\big)\big) .
\end{gather*}
We define a homomorphism $d_{\tilde{\nabla}} \colon \widetilde{\mathcal{F}}^0 \rightarrow \widetilde{\mathcal{F}}^1$ as $s \mapsto \tilde{\nabla} \circ s - s \circ \tilde{\nabla} $ and we define a homomorphism $d_{\tilde{\psi}} \colon \Theta_{\mathcal{C}_{\widehat{M}}/ \widehat{M}} \big({-}D\big(\tilde{\boldsymbol{t}}\big)\big)\rightarrow \Omega_{\mathcal{C}_{\widehat{M}}/ \widehat{M}}^{\otimes 2} \big(D\big(\tilde{\boldsymbol{t}}\big)\big)$ as follows. Take an affine open covering $\{ U_{\alpha} \}$ of $\mathcal{C}_{\widehat{M}}$ such that we can take trivializations of $\Theta_{\mathcal{C}_{\widehat{M}}/ \widehat{M}} \big({-}D\big(\tilde{\boldsymbol{t}}\big)\big)$ and $\Omega_{\mathcal{C}_{\widehat{M}}/ \widehat{M}}^{\otimes 2} \big(D\big(\tilde{\boldsymbol{t}}\big)\big)$ on each $U_{\alpha}$. For an element $a \partial/\partial f_{\alpha} \in \Theta_{\mathcal{C}_{\widehat{M}}/ \widehat{M}} \big({-}D\big(\tilde{\boldsymbol{t}}\big)\big) (U_{\alpha})$, we define a homomorphism on $U_{\alpha}$ by
\begin{gather}\label{dpsi Ui}
 a \frac{\partial}{\partial f_{\alpha}}\longmapsto \left( \frac{\partial \psi_{U_{\alpha}}}{\partial f_{\alpha}}a+ 2 \psi_{U_{\alpha}} \frac{\partial a}{\partial f_{\alpha}} \right) d f_{\alpha} \otimes d f_{\alpha}
\in \Omega_{\mathcal{C}_{\widehat{M}}/ \widehat{M}}^{\otimes 2} \big(D\big(\tilde{\boldsymbol{t}}\big)\big)(U_{\alpha}),
\end{gather}
where $\tilde{\psi}|_{U_{\alpha}} = \psi_{U_{\alpha}} d f_{\alpha} \otimes d f_{\alpha}$. By the homomorphism on each $U_{\alpha}$, we can define a homomor\-phism~$d_{\tilde{\psi}}$. We define a complex $\mathcal{F}^{\bullet}$ by the differential $d_{\mathcal{F}^{\bullet}} = (d_{\tilde{\nabla}} ,d_{\tilde{\psi}} ) \circ \big(\mathrm{Id} - \iota\big(\tilde{\nabla}\big) \circ \mathrm{symb}_1,\mathrm{symb}_1\big)$:
\begin{gather*}
\xymatrix{\mathcal{F}^0 \ar[d]_-{(\mathrm{Id} -
\iota(\tilde{\nabla}) \circ \mathrm{symb}_1, \mathrm{symb}_1 )} \ar[rd]^{d_{\mathcal{F}^{\bullet}}} & \\
 \widetilde{\mathcal{F}}^0 \oplus \Theta_{\mathcal{C}_{\widehat{M}}/ \widehat{M}} \big({-}D\big(\tilde{\boldsymbol{t}}\big)\big)
 \ar[r]^-{(d_{\tilde{\nabla}},d_{\tilde{\psi}})}
&\widetilde{\mathcal{F}}^1 \oplus \Omega_{\mathcal{C}_{\widehat{M}}/ \widehat{M}}^{\otimes 2} \big(D\big(\tilde{\boldsymbol{t}}\big)\big) \rlap{.}}
\end{gather*}

\begin{Proposition}\label{Tangent sheaf Cech 1}We put $\widehat{M}=\widehat{M}_{\mathcal{C}/\tilde{M}_{g,n}}^{\boldsymbol{\alpha}}\big(\tilde{\boldsymbol{t}},r,e\big)$ and $M=M_{\mathcal{C}/\tilde{M}_{g,n}}^{\boldsymbol{\alpha}}\big(\tilde{\boldsymbol{t}},r,e\big)$. Let $\mathcal{F}_{M}^{0}$, $\widetilde{\mathcal{F}}_{M}^{0}$, and~$\widetilde{\mathcal{F}}_{M}^{1}$ be the pull-backs of $\mathcal{F}^{0}$, $\widetilde{\mathcal{F}}^{0}$, and $\widetilde{\mathcal{F}}^{1}$ by the natural immersion $\mathcal{C}_M \rightarrow \mathcal{C}_{\widehat{M}}$, respectively. There exist canonical isomorphisms
\begin{gather*}
\hat{\varsigma} \colon \ \Theta_{\widehat{M}/N^{(n)}_r(e)}\xrightarrow{\, \sim \, } \bold{R}^1(\pi_{\widehat{M}})_*(\mathcal{F}^{\bullet}), \\
\tilde{\varsigma} \colon \ \Theta_{M/N^{(n)}_r(e)}\xrightarrow{\, \sim \, } \bold{R}^1(\pi_{M})_*\big(\mathcal{F}_{M}^{0} \rightarrow \widetilde{\mathcal{F}}_{M}^1\big), \qquad \text{and}\\
\varsigma \colon \ \Theta_{M/ (\tilde{M}_{g,n} \times N^{(n)}_r(e))} \xrightarrow{\, \sim \, } \bold{R}^1(\pi_{M})_*\big(\widetilde{\mathcal{F}}_M^{0} \rightarrow \widetilde{\mathcal{F}}_M^1\big),
\end{gather*}
where $\pi_{\widehat{M}} \colon \mathcal {C}_{\widehat{M}} \rightarrow \widehat{M}$ and $\pi_{M} \colon \mathcal{C}_{M} \rightarrow M$ are the natural morphisms.
\end{Proposition}
\begin{proof}We show the existence of the isomorphism $\hat{\varsigma}$. For the existence of the isomorphisms $\tilde{\varsigma}$ and~$\varsigma$, see the proof of \cite[Proposition~3.2]{Komyo} and the proof of \cite[Theorem~2.1]{Inaba}. We take an affine open set $\widehat{U}\subset \widehat{M}$. Let $\big(\tilde{E},\tilde{\nabla} ,\big\{ \tilde{l}^{(i)}_j\big\}, \tilde{\psi}\big)$ be the family on $ \mathcal{C} \times_{\tilde{M}_{g,n}} \widehat{U}$. We take an affine open covering $\mathcal{C}_{\widehat{U}} = \bigcup_{\alpha} U_{\alpha}$ such that $\phi_{\alpha} \colon \tilde{E}|_{U_{\alpha}} \xrightarrow{\sim} \mathcal{O}^{\oplus r}_{U_{\alpha}}$ for any $\alpha$, $\sharp\{ i \,|\, \tilde{t}_i |_{\mathcal{C}_U} \cap U_{\alpha} \neq \varnothing \} \le 1$ for any $\alpha$ and $\sharp\{ \alpha \,|\, \tilde{t}_i |_{\mathcal{C}_U} \cap U_{\alpha} \neq \varnothing \} \le 1$ for any~$i$. Take a relative tangent vector field $v \in \Theta_{\widehat{M}/N^{(n)}_r(e)} \big(\widehat{U}\big)$. The field $v$ corresponds to a member $\big( (C_{\epsilon},\boldsymbol{t}_{\epsilon},\psi_{\epsilon}), \big(E_{\epsilon},\nabla_{\epsilon}, \big\{ (l_{\epsilon})^{(i)}_j \big\} \big)\big) \in \widehat{M}(\operatorname{Spec} \mathcal{O}_{\widehat{U}}[\epsilon])$ such that $\big( (C_{\epsilon},\boldsymbol{t}_{\epsilon},\psi_{\epsilon} ), \big(E_{\epsilon},\nabla_{\epsilon},\big\{ (l_{\epsilon})^{(i)}_j \big\} \big)\big)\otimes \mathcal{O}_{\widehat{U}}[\epsilon]/(\epsilon) \cong \big(\big(\mathcal{C}_{\widehat{U}}, \boldsymbol{t}_{\widehat{U}} ,\tilde{\psi}\big), \big(\tilde{E},\tilde{\nabla},\big\{ \tilde{l}^{(i)}_j \big\}\big)\big)$, where $\mathcal{O}_{\widehat{U}}[\epsilon]= \mathcal{O}_{\widehat{U}}[t]/\big(t^2\big)$. Here,
\begin{itemize}\itemsep=0pt
\item $\psi_{\epsilon} \in H^0 \big(C_{\epsilon} , \Omega^{\otimes 2}_{C_{\epsilon}/\widehat{U}} \big(\log\big(D\big(\tilde{\boldsymbol{t}}\big)_{\mathcal{O}_{\widehat{U}} [\epsilon]}\big) \big)\big)$, and
\item $\nabla_{\epsilon} \colon E_{\epsilon} \rightarrow E_{\epsilon} \otimes \Omega^1_{C_{\epsilon} / \widehat{U}} \big(\log\big(D\big(\tilde{\boldsymbol{t}}\big)_{\mathcal{O}_{\widehat{U}} [\epsilon]}\big) \big)$ is a~connection,
\end{itemize}
where we define the sheaf $\Omega^1_{C_{\epsilon}/\widehat{U}} \big(\log\big(D\big(\tilde{\boldsymbol{t}}\big)_{\mathcal{O}_{\widehat{U}} [\epsilon]}\big) \big)$ as the coherent subsheaf of $\Omega^1_{C_\epsilon/\widehat{U}} \big(D\big(\tilde{\boldsymbol{t}}\big)_{\mathcal{O}_{\widehat{U}} [\epsilon]}\big)$ locally generated by $f^{-1} df$ and $d \epsilon$ for a local defining equation~$f$ of $D\big(\tilde{\boldsymbol{t}}\big)_{\mathcal{O}_{\widehat{U}} [\epsilon]}$ which is the pull-back of $D(\tilde{\boldsymbol{t}})$ by the morphism $C_{\epsilon} \rightarrow \mathcal{C}_{\widehat{U}} \rightarrow \mathcal{C}$. Set $U_{\alpha}^{\epsilon}:= U_{\alpha}\times \operatorname{Spec} \mathcal{O}_{\widehat{U}}[\epsilon]$. Let
\begin{gather}\label{muAB}
\mu_{\alpha\beta}(\epsilon) \colon \ U_{\alpha\beta}\times \operatorname{Spec} \mathcal{O}_{\widehat{U}}[\epsilon] \xrightarrow{\, \sim\, } U_{\alpha\beta}\times \operatorname{Spec} \mathcal{O}_{\widehat{U}}[\epsilon]
\end{gather}
be an isomorphism associated to the first-order deformation $C_{\epsilon}$ of $\mathcal{C}_{\widehat{U}}$. The isomorphism $\mu_{\alpha\beta}(\epsilon)$ satisfies
\begin{gather*}
\mu_{\alpha\beta}(\epsilon)^*(\epsilon) = \epsilon , \qquad \mu_{\alpha\beta}(\epsilon)^*(f) = f +\epsilon d_{\alpha\beta} f, \qquad \text{for} \quad f \in \mathcal{O}_{U_{\alpha\beta}},
\end{gather*}
for some $d_{\alpha\beta} \in \Theta_{\mathcal{C}_{\widehat{U}}}(-D)(U_{\alpha\beta})$. We describe $d_{\alpha\beta}$ as $d_{\alpha\beta}= \frac{\partial \mu_{\alpha\beta}(\epsilon)}{\partial \epsilon} \frac{\partial}{\partial f_{\alpha}} \in \Theta_{\mathcal{C}_{\widehat{U}}}(-D)(U_{\alpha\beta})$. Here, $f_{\alpha}$ is a local defining equation of $\tilde{t}_i |_{\mathcal{C}_{\widehat{U}}} \cap U_{\alpha}$.
Set $\phi_{\alpha}^{\epsilon} \colon E_{\epsilon}|_{U^{\epsilon}_{\alpha}} \cong \mathcal{O}^{\oplus r}_{U_{\alpha}^{\epsilon}}$. There is an isomorphism
\begin{gather*}
\varphi_{\alpha} \colon \ E_{\epsilon}|_{U^{\epsilon}_{\alpha}} \xrightarrow[\, \sim\,]{\, \phi_{\alpha}^{\epsilon}\, } \mathcal{O}^{\oplus r}_{U_{\alpha}^{\epsilon}}\xrightarrow[\, \sim\,]{\, \phi_{\alpha}^{-1}\, } \tilde{E}|_{U_{\alpha}} \otimes \mathcal{O}_{\widehat{U}}[\epsilon]
\end{gather*}
such that $\varphi_{\alpha}\otimes \mathcal{O}_{\widehat{U}}[\epsilon]/(\epsilon ) \colon E_{\epsilon}\otimes \mathcal{O}_{\widehat{U}} [\epsilon]/(\epsilon)|_{U_{\alpha}}\xrightarrow{\sim} \tilde{E}|_{U_{\alpha}}\otimes \mathcal{O}_{\widehat{U}} [\epsilon]/(\epsilon)=\tilde{E}|_{U_{\alpha}}$ is the given isomorphism and that $\varphi_{\alpha}|_{t_i\otimes \mathcal{O}_{\widehat{U}}[\epsilon]} ((l_{\epsilon})_j^{(i)})= \tilde{l}_j^{(i)}|_{U_{\alpha} \times \operatorname{Spec} \mathcal{O}_{\widehat{U}}[\epsilon]}$ if $\tilde{t}_i |_{\mathcal{C}_{\widehat{U}}}\cap U_{\alpha} \neq \varnothing$. Put
\begin{gather*}
\theta_{\alpha\beta}(\epsilon) \colon \ \mathcal{O}^{\oplus r}_{U_{\alpha\beta}^{\epsilon}} \xrightarrow[\, \sim\,]{(\phi^{\epsilon}_{\beta})^{-1}|_{U_{\alpha\beta}^{\epsilon} }} E_{\epsilon}|_{U^{\epsilon}_{\alpha\beta}}
\xrightarrow[\, \sim\,]{\phi_{\alpha}^{\epsilon}|_{U^{\epsilon}_{\alpha\beta}}} \mathcal{O}^{\oplus r}_{U_{\alpha\beta}^{\epsilon}} ,
\end{gather*}
which is an element of $\mathcal{E}{\rm nd} \big(\mathcal{O}^{\oplus r}_{U_{\alpha\beta}^{\epsilon}}\big)(U_{\alpha\beta}^{\epsilon})$. We denote $\theta_{\alpha\beta}(\epsilon)$ by
\begin{gather*}
\theta_{\alpha\beta}(\epsilon)= \tilde{\theta}_{\alpha\beta} + \epsilon \frac{\partial \theta_{\alpha\beta}(\epsilon)}{\partial \epsilon},
\qquad \text{where} \quad \tilde{\theta}_{\alpha\beta}, \frac{\partial \theta_{\alpha\beta}(\epsilon)}{\partial \epsilon} \in \mathcal{E}{\rm nd} \big(\mathcal{O}^{\oplus r}_{U_{\alpha\beta}}\big)(U_{\alpha\beta}).
\end{gather*}
Set
\begin{gather*}
\eta_{\alpha\beta} := \frac{\partial \theta_{\alpha\beta}(\epsilon)}{\partial \epsilon} \big(\tilde{\theta}_{\alpha\beta}\big)^{-1}\in \mathcal{E}{\rm nd} \big(\mathcal{O}^{\oplus r}_{U_{\alpha\beta}}\big)(U_{\alpha\beta}).
\end{gather*}
We define elements $u_{\alpha\beta} \in \mathcal{F}^0(U_{\alpha\beta})$ and $(v_{\alpha} ,w_{\alpha})\in \mathcal{F}^1(U_{\alpha})$ by
\begin{gather*}
u_{\alpha\beta} := \ ({\phi}_{\alpha}|_{U_{\alpha\beta}})^{-1} \circ \epsilon (d_{\alpha\beta}+ \eta_{\alpha\beta} ) \circ {\phi}_{\alpha} |_{U_{\alpha\beta}}, \\
 v_{\alpha} := (\varphi_{\alpha} \otimes \mathrm{id}) \circ \nabla_{\epsilon}|_{U_{\alpha}^{\epsilon}} \circ \varphi_{\alpha}^{-1} - \tilde{\nabla}|_{U_{\alpha}^{\epsilon}} \qquad \text{mod $d \epsilon$},\\
w_{\alpha} := \psi_{\epsilon}|_{U_{\alpha}^{\epsilon}} - \tilde{\psi}|_{U_{\alpha}^{\epsilon}} \qquad \text{mod $d \epsilon$},
\end{gather*}
respectively. We can see that
\begin{gather*}
u_{\beta\gamma} - u_{\alpha\gamma } + u_{\alpha\beta}=0, \qquad \text{and} \qquad d_{\mathcal{F}^{\bullet}} (u_{\alpha\beta}) = (v_{\beta},w_{\beta}) -( v_{\alpha}, w_{\alpha}).
\end{gather*}
Then $ [ \{ u_{\alpha\beta} \}, \{ ( v_{\alpha}, w_{\alpha})\}]$ determines an element $\sigma_{\widehat{U}}(v)$ of ${\mathbf H}^1\big(\mathcal{F}_{\widehat{U}}^{\bullet}\big)$. We can check that $v \mapsto \sigma_{\widehat{U}}(v)$ determines an isomorphism
\begin{gather*}
\Theta_{\widehat{M}/N^{(n)}_r(e)} \big(\widehat{U}\big) \xrightarrow{\, \sim \, } {\mathbf H}^1\big(\mathcal{F}_{\widehat{U}}^{\bullet}\big), \qquad v \longmapsto \sigma_{\widehat{U}}(v).
\end{gather*}
We denote by $\hat{\varsigma}_{\widehat{U}}$ this isomorphism. The isomorphism $\hat{\varsigma}_{\widehat{U}}$ induces the desired isomorphism $\hat{\varsigma}$.
\end{proof}

We describe the analytic tangent sheaf in terms of the hypercohomology of a certain analytic complex. Let $\boldsymbol{\nu}$ be an element of $N^{(n)}_r(e)$. Put $\widehat{M}_{\boldsymbol{\nu}} = \widehat{M}_{\mathcal{C}/\tilde{M}_{g,n}}^{\boldsymbol{\alpha}}\big(\tilde{\boldsymbol{t}},r,e\big)_{\boldsymbol{\nu}}$, which is the fiber of $\boldsymbol{\nu}$ under $\widehat{M}_{\mathcal{C}/\tilde{M}_{g,n}}^{\boldsymbol{\alpha}}\big(\tilde{\boldsymbol{t}},r,e\big) \rightarrow N^{(n)}_r(e)$. Let $j \colon \mathcal{C}_{\widehat{M}_{\boldsymbol{\nu}}} \setminus \big\{\tilde{t}_1 ,\ldots, \tilde{t}_n\big\}_{\widehat{M}_{\boldsymbol{\nu}}} \rightarrow \mathcal{C}_{\widehat{M}_{\boldsymbol{\nu}}}$ be the canonical inclusion. Let $\widehat{\boldsymbol{V}} :=\operatorname{Ker} \tilde{\nabla}^{{\rm an}}|_{\mathcal{C}_{\widehat{M}_{\boldsymbol{\nu}}} \setminus \big\{\tilde{t}_1 ,\ldots, \tilde{t}_n\big\}_{\widehat{M}_{\boldsymbol{\nu}}} }$ be the locally constant sheaf of the locally free $(\pi_{\widehat{M}_{\boldsymbol{\nu}}}\circ j)^{-1}\mathcal{O}_{\widehat{M}_{\boldsymbol{\nu}}}$-module associated to the relative analytic connection $\tilde{\nabla}^{{\rm an}}$ on $\mathcal{C}_{\widehat{M}_{\boldsymbol{\nu}}}\setminus \big\{\tilde{t}_1 ,\ldots, \tilde{t}_n\big\}_{\widehat{M}_{\boldsymbol{\nu}}}$, where $\pi_{\widehat{M}_{\boldsymbol{\nu}}} \colon \mathcal{C}_{\widehat{M}_{\boldsymbol{\nu}}} \rightarrow \widehat{M}_{\boldsymbol{\nu}}$ is the natural map.

Assume that $\boldsymbol{\nu}$ is generic. We define a complex $\big(\widehat{\mathcal{F}}^{\bullet}\big)^{{\rm an}}$ by
\begin{gather*}\label{complexFhatan}
\big(\widehat{\mathcal{F}}^{\bullet}\big)^{{\rm an}} \colon \ j_* \big( \mathcal{E}{\rm nd} \big(\widehat{\boldsymbol{V}}\big) \big)\oplus
\Theta_{\mathcal{C}_{\widehat{M}_{\boldsymbol{\nu}}}/ \widehat{M}_{\boldsymbol{\nu}}} \big({-}D\big(\tilde{\boldsymbol{t}}\big)\big)
\xrightarrow{\ d_{\tilde{\psi}} \circ\, \mathrm{pr}_2 \ } \Omega^{\otimes 2}_{\mathcal{C}_{\widehat{M}_{\boldsymbol{\nu}}}/ \widehat{M}_{\boldsymbol{\nu}}} \big(D\big(\tilde{\boldsymbol{t}}\big)\big),
\end{gather*}
where $\mathrm{pr}_2$ is the second projection.
We have the following commutative diagram
\begin{gather*}
\xymatrix{
j_* \big( \mathcal{E}{\rm nd} \big(\widehat{\boldsymbol{V}}\big) \big)\oplus
 \Theta_{\mathcal{C}_{\widehat{M}_{\boldsymbol{\nu}}}/ \widehat{M}_{\boldsymbol{\nu}}} \big({-}D\big(\tilde{\boldsymbol{t}}\big)\big)
 \ar[r]^-{d_{\tilde{\psi}} \circ\, \mathrm{pr}_2} \ar[d] &
\Omega^{\otimes 2}_{\mathcal{C}_{\widehat{M}_{\boldsymbol{\nu}}}/ \widehat{M}_{\boldsymbol{\nu}}} \big(D\big(\tilde{\boldsymbol{t}}\big)\big)\ar[d] \\
\big(\mathcal{F}^0\big)^{{\rm an}} \ar[r]^-{(d_{\mathcal{F}^{\bullet}})^{{\rm an}}} & \big(\mathcal{F}^1\big)^{{\rm an}}.}
\end{gather*}
We can show that the homomorphism $\operatorname{Ker} d_{\tilde{\nabla}^{{\rm an}}}|_{\mathcal{C}_{\widehat{M}_{\boldsymbol{\nu}}}} \rightarrow j_* \big( \mathcal{E}{\rm nd} \big(\widehat{\boldsymbol{V}}\big)\big)$ is an isomorphism and the homomorphism $d_{\tilde{\nabla}^{{\rm an}}} \colon \big(\widetilde{\mathcal{F}}^0\big)^{{\rm an}} \rightarrow \big(\widetilde{\mathcal{F}}^1\big)^{{\rm an}}$ is surjective as in the proof of \cite[Proposition~7.3]{Inaba}. Then we have the following proposition.
\begin{Proposition}If $\boldsymbol{\nu}$ is generic, then we have
\begin{gather*}
\bold{R}^1(\pi_{\widehat{M}_{\boldsymbol{\nu}}})_*\big((\mathcal{F}^{\bullet})^{{\rm an}}\big)
\xrightarrow{\, \sim \, } \bold{R}^1(\pi_{\widehat{M}_{\boldsymbol{\nu}}})_*\big(\big(\widehat{\mathcal{F}}^{\bullet}\big)^{{\rm an}}\big),
\end{gather*}
where $\pi_{\widehat{M}_{\boldsymbol{\nu}}} \colon \mathcal{C}_{\widehat{M}_{\boldsymbol{\nu}}} \rightarrow \widehat{M}_{\boldsymbol{\nu}}$ is the natural map.
\end{Proposition}

\subsection{Isomonodromic deformations}\label{SS Isomonod defor}
Let $\boldsymbol{\nu}$ be an element of $N^{(n)}_r(e)$. Put $M_{\boldsymbol{\nu}}=M_{\mathcal{C}/\tilde{M}_{g,n}}^{\boldsymbol{\alpha}}\big(\tilde{\boldsymbol{t}},r,e\big)_{\boldsymbol{\nu}}$ which is the fiber of $\boldsymbol{\nu}$ under $M_{\mathcal{C}/\tilde{M}_{g,n}}^{\boldsymbol{\alpha}}\big(\tilde{\boldsymbol{t}},r,e\big) \rightarrow N^{(n)}_r(e)$. Let $j \colon \mathcal{C}_{M_{\boldsymbol{\nu}}} \setminus \big\{\tilde{t}_1 ,\ldots, \tilde{t}_n\big\}_{M_{\boldsymbol{\nu}}}\rightarrow \mathcal{C}_{M_{\boldsymbol{\nu}}}$ be the canonical inclusion. Let $\operatorname{Ker} \tilde{\nabla}^{{\rm an}}|_{\mathcal{C}_{M_{\boldsymbol{\nu}}} \setminus \{\tilde{t}_1 ,\ldots, \tilde{t}_n\}_{M_{\boldsymbol{\nu}}} }$ be the locally constant sheaf of the locally free $(\pi_{M_{\boldsymbol{\nu}}}\circ j)^{-1}\mathcal{O}_{M_{\boldsymbol{\nu}}}$-module associated to the relative analytic connection $\tilde{\nabla}^{{\rm an}}$ on $\mathcal{C}_{M_{\boldsymbol{\nu}}}\setminus \big\{\tilde{t}_1 ,\ldots, \tilde{t}_n\big\}_{M_{\boldsymbol{\nu}}}$, where \smash{$\pi_{M_{\boldsymbol{\nu}}} \colon \mathcal{C}_{M_{\boldsymbol{\nu}}} \rightarrow M_{\boldsymbol{\nu}}$} is the natural map.

\begin{Definition}For $\pi_{\boldsymbol{\nu}} \colon M_{\boldsymbol{\nu}}\rightarrow \tilde{M}_{g,n}$, we say a complex foliation $\mathcal{F}$ is a \textit{foliation determined by the isomonodromic deformations} if
\begin{itemize}\itemsep=0pt
\item[(1)] $\mathcal{F}$ is transverse to each fiber $(M_{\boldsymbol{\nu}})_t = \pi_{\boldsymbol{\nu}}^{-1} (t)$, $t \in \tilde{M}_{g,n}$, and
\item[(2)] for each leaf $l$ on $M_{\boldsymbol{\nu}}$, the restriction of the local system $j_*\big(\operatorname{Ker} \tilde{\nabla}^{{\rm an}}|_{\mathcal{C}_{M_{\boldsymbol{\nu}}} \setminus \big\{\tilde{t}_1 ,\ldots, \tilde{t}_n\big\} }\big)|_{\mathcal{C}\times_{ \tilde{M}_{g,n}}l}$ is constant.
\end{itemize}
\end{Definition}

Let $\mu \colon \pi_{\boldsymbol{\nu}}^* \Theta_{\tilde{M}_{g,n}} \rightarrow R^1 ( \pi_{M_{\boldsymbol{\nu}}})_*\big(\Theta_{\mathcal{C}_{M_{\boldsymbol{\nu}}} / M_{\boldsymbol{\nu}}} \big({-}D\big(\tilde{\boldsymbol{t}}\big)\big) \big)$ be the Kodaira--Spencer map, where $\pi_{M_{\boldsymbol{\nu}}} \colon$ $\mathcal{C}_{M_{\boldsymbol{\nu}}} \rightarrow M_{\boldsymbol{\nu}}$ is the natural morphism.
We define a splitting $\mathfrak{D}$ of the tangent map $\Theta_{M_{\boldsymbol{\nu}}} \rightarrow \pi_{\boldsymbol{\nu}}^* \big(\Theta_{\tilde{M}_{g,n}}\big)$ as follows
\begin{gather*}
\mathfrak{D} \colon \ \pi_{\boldsymbol{\nu}}^* \big(\Theta_{\tilde{M}_{g,n}}\big) \longrightarrow \Theta_{M_{\boldsymbol{\nu}}} \cong
 \bold{R}^1(\pi_{M_{\boldsymbol{\nu}}})_* \big(\mathcal{F}_{M_{\boldsymbol{\nu}}}^{0} \rightarrow \widetilde{\mathcal{F}}_{M_{\boldsymbol{\nu}}}^1 \big),
\qquad v \longmapsto \big[\big\{ \iota\big(\tilde{\nabla}\big)(d_{\alpha\beta}) \big\}, \{ 0\}\big],
\end{gather*}
where $[\{ d_{\alpha\beta} \}]$ is a description of $\mu(v)$ by the \v{C}ech cohomology. Here, we take an affine open covering $\{ U_{\alpha} \}$.

\begin{Proposition}[{\cite[Section~6]{Hurt}, \cite[Section~8]{Inaba}, \cite[Section~3.3]{Komyo}}] The subsheaf $\mathfrak{D} \big(\pi^* \big(\Theta_{\tilde{M}_{g,n}}\big)\big)$ determines the foliation determined by the isomonodromic deformations.
\end{Proposition}

We can take a natural lift $\widehat{\mathfrak{D}} \colon \hat{\pi}_{\boldsymbol{\nu}}^* \big(\Theta_{T^* \tilde{M}_{g,n}}\big) \rightarrow \Theta_{\widehat{M}_{\boldsymbol{\nu}}}$ of $\mathfrak{D} \colon \pi_{\boldsymbol{\nu}}^* \big(\Theta_{\tilde{M}_{g,n}}\big) \rightarrow \Theta_{M_{\boldsymbol{\nu}}}$ as follows. We define a complex $\mathcal{G}^{\bullet}$ by
\begin{gather}\label{complexG}
 \Theta_{\mathcal{C}_{\widehat{M}_{\boldsymbol{\nu}}}/\widehat{M}_{\boldsymbol{\nu}}}\big({-}D\big(\tilde{\boldsymbol{t}}\big)\big) =:\mathcal{G}^0
 \xrightarrow{\ d_{\tilde{\psi}}\ }
\mathcal{G}^1 :=\Omega^{\otimes 2}_{\mathcal{C}_{\widehat{M}_{\boldsymbol{\nu}}}/\widehat{M}_{\boldsymbol{\nu}}}\big(D\big(\tilde{\boldsymbol{t}}\big)\big),
\end{gather}
where $d_{\tilde{\psi}}$ is defined by (\ref{dpsi Ui}). Then we can show that $\hat{\pi}_{\boldsymbol{\nu}}^* \Theta_{T^* \tilde{M}_{g,n}} \cong \bold{R}^1 (\pi_{\widehat{M}_{\boldsymbol{\nu}}})_* (\mathcal{G}^{\bullet})$. We define a~lift $\widehat{\mathfrak{D}} \colon \hat{\pi}_{\boldsymbol{\nu}}^* \Theta_{T^*\tilde{M}_{g,n}}\rightarrow \Theta_{\widehat{M}_{\boldsymbol{\nu}}} $ of $\mathfrak{D}$ by the following homomorphism
\begin{gather}\label{def hatD}
\widehat{\mathfrak{D}} \colon \ {\mathbf H}^1\big(\mathcal{G}^{\bullet}_{\widehat{U}}\big) \longrightarrow {\mathbf H}^1 \big(\mathcal{F}_{\widehat{U}}^{\bullet}\big), \qquad
[\{ d_{\alpha\beta} \} ,\{ w_{\alpha}\} ]\longmapsto \big[\big\{ \iota \big(\tilde{\nabla}\big)(d_{\alpha\beta} ) \big\},\{(0, w_{\alpha}) \}\big].
\end{gather}

\subsection{Symplectic structure}\label{SS Sympl and Hamilt}

First, we recall the canonical symplectic structure $\omega_{\tilde{M}_{g,n}}$ on $T^*\tilde{M}_{g,n}$.
Let $U$ be an affine open set of $T^*\tilde{M}_{g,n}$
and let $(\mathcal{C}_U,\tilde{\psi})$ be a family of curves and quadratic differentials on $U$.
Let $\psi_{\alpha} df_{\alpha}^{\otimes 2}$ be the restriction of $\tilde{\psi}$
on an affine open set $U_{\alpha} \subset \mathcal{C}_U$.
Let $\mu_{\alpha \beta}$ be the isomorphism (\ref{muAB}): $f_\alpha= \mu_{\alpha \beta}(f_{\beta})$.
We define a 1-form $\theta_{\tilde{M}_{g,n}}$ on $T^*\tilde{M}_{g,n}$ by
\begin{gather*}
\theta_{\tilde{M}_{g,n}} \colon \ {\mathbf H}^1\big(\mathcal{G}^{\bullet}_{U}\big) \longrightarrow H^1 \big( \Omega^1_{\mathcal{C}_{U}/U} \big) , \qquad
[ \{ d_{\alpha\beta} \} ,\{ w_{\alpha}\} ] \longmapsto \left[ \left\{ d_{\beta\alpha } \psi_{\alpha} \frac{\partial \mu_{\alpha\beta}}{\partial f_{\beta}} df_{\alpha} \right\} \right],
\end{gather*}
where $\mathcal{G}_U$ is the complex $d_{\tilde{\psi}} \colon\Theta_{\mathcal{C}_U /U} \big({-}D\big(\tilde{\boldsymbol{t}}\big)\big) \rightarrow \Omega^{\otimes 2}_{\mathcal{C}_U /U}\big(D\big(\tilde{\boldsymbol{t}}\big)\big) $. The 1-form $\theta_{\tilde{M}_{g,n}}$ is the canonical 1-form on the cotangent bundle $T^*\tilde{M}_{g,n}$. Let $d \theta_{\tilde{M}_{g,n}}$ be the exterior differential of $\theta_{\tilde{M}_{g,n}}$. The 2-form $d \theta_{\tilde{M}_{g,n}}$ gives the symplectic form on the cotangent bundle $T^*\tilde{M}_{g,n}$.

\begin{Proposition} Let $v=[ (\{d_{\alpha\beta}\}, \{w_\alpha\}) ] $ and $v'=[ (\{d'_{\alpha\beta}\}, \{w'_\alpha\}) ]$ be elements of ${\mathbf H}^1\big(\mathcal{G}_U^{\bullet}\big)$. The pairing
\begin{gather*}
{\mathbf H}^1\big(\mathcal{G}_U^{\bullet}\big) \otimes {\mathbf H}^1\big(\mathcal{G}_U^{\bullet}\big) \longrightarrow {\mathbf H}^2 \big(\Omega^{\bullet}_{\mathcal{C}_U/U}\big),\\
v\otimes w \longmapsto [ \{ 2 \cdot d_{\beta\alpha} \circ d'_{\beta\gamma} \circ \psi_{\beta} \}, \{ - d_{\beta\alpha} \circ w_{\beta}'- w_{\alpha} \circ d'_{\alpha\beta} \} ]
\end{gather*}
coincides with the symplectic form $d \theta_{\tilde{M}_{g,n}}$.
\end{Proposition}

\begin{proof}Let $D_{v} \colon \mathcal{O}_{U_{\alpha\beta}} \rightarrow \mathcal{O}_{U_{\alpha\beta}}$ be a derivation corresponding to $v$. We compute the 2-form $d \theta_{\tilde{M}_{g,n}} (v, v')$ as follows
\begin{gather*}
D_{v} \theta_{\tilde{M}_{g,n}}(v') -D_{v'} \theta_{\tilde{M}_{g,n}}(v) + \theta_{\tilde{M}_{g,n}}([v,v']) \\
\qquad{} = D_{v'} (\mu_{\beta \alpha}) D_v \left( \psi_{\alpha} \frac{\partial \mu_{\alpha\beta}}{\partial f_{\beta}} \right) df_{\alpha}- D_{v} (\mu_{\beta \alpha}) D_{v'}
\left( \psi_{\alpha} \frac{\partial \mu_{\alpha\beta}}{\partial f_{\beta}} \right) df_{\alpha} \\
\qquad{} = D_{v'} (\mu_{\beta \alpha}) \psi_{\alpha} \frac{\partial D_v(\mu_{\alpha \beta})}{\partial f_{\beta}} df_{\alpha}- D_{v} (\mu_{ \beta \alpha}) \psi_{\alpha}
\frac{\partial D_{v'}(\mu_{\alpha\beta})}{\partial f_{\beta}} df_\alpha \\
\qquad\quad {} + D_{v'} (\mu_{\beta \alpha}) \frac{\partial \mu_{\alpha\beta}}{\partial f_{\beta}} D_v( \psi_{\alpha} df_\alpha )- D_{v} (\mu_{\beta \alpha}) \frac{\partial \mu_{\alpha\beta}}{\partial f_{\beta}}
D_{v'}( \psi_{\alpha} df_\alpha ) \\
\qquad{} = - d'_{\alpha\beta} \psi_\alpha \frac{\partial d_{\alpha\beta}}{\partial f_{\beta}} df_{\beta}
+ d_{\alpha\beta} \psi_\alpha \frac{\partial d'_{\alpha\beta}}{\partial f_{\beta}} df_{\beta}
- d'_{\alpha\beta} w_\alpha \frac{\partial \mu_{\alpha\beta}}{\partial f_{\beta}} df_{\beta}
+ d_{\alpha\beta} w'_\alpha \frac{\partial \mu_{\alpha\beta}}{\partial f_{\beta}} df_{\beta}.
\end{gather*}
We add the exterior differential of $d'_{\alpha \beta} d_{\alpha \beta} \psi_\alpha$ to the formula above
\begin{gather*}
 - d'_{\alpha \beta} \psi_\alpha \frac{\partial d_{\alpha \beta}}{\partial f_\beta} df_\beta+ d_{\alpha \beta} \psi_\alpha \frac{\partial d'_{\alpha\beta}}{\partial f_\beta} df_\beta
- d'_{\alpha \beta} w_\alpha \frac{\partial \mu_{\alpha \beta}}{\partial f_\beta} df_\beta+ d_{\alpha \beta} w_\alpha \frac{\partial \mu_{\alpha \beta}}{\partial f_\beta} df_\beta
+ d\left( d'_{\alpha \beta} d_{\alpha \beta} \psi_\alpha \right)\\
\qquad{} = d_{\alpha \beta} \left( d'_{\alpha \beta} d \psi_\alpha+ 2 \psi_\alpha \frac{\partial d'_{\alpha \beta}}{\partial f_\beta} df_\beta \right)- d'_{\alpha \beta} w_\alpha df_\alpha + d_{\alpha \beta} w'_\alpha df_\alpha \\
\qquad{} = - d_{\beta\alpha} w'_\beta df_\beta - d'_{\alpha \beta} w_\alpha df_\alpha.
\end{gather*}
By the isomorphism $H^1 \big( \Omega^1_{\mathcal{C}_{U}/U} \big) \cong {\mathbf H}^2 \big( \Omega^{\bullet}_{\mathcal{C}_{U}/U} \big)$, we have this proposition.
\end{proof}

\begin{Proposition}\label{Prop symplectic}
Take a point $\boldsymbol{\nu} \in N^{(n)}_r(e)$.
Let $\widehat{M}_{\mathcal{C}/\tilde{M}_{g,n}}^{\boldsymbol{\alpha}}(\tilde{\boldsymbol{
t}},r,e)_{\boldsymbol{\nu}}$
be the fiber of $\boldsymbol{\nu}$ under the composition
$\widehat{M}_{\mathcal{C}/\tilde{M}_{g,n}}^{\boldsymbol{\alpha}}\big(\tilde{\boldsymbol{t}},r,e\big)_{\boldsymbol{\nu}} \rightarrow N^{(n)}_r(e)$.
Then the fiber $\widehat{M}_{\mathcal{C}/\tilde{M}_{g,n}}^{\boldsymbol{\alpha}}\big(\tilde{\boldsymbol{t}},r,e\big)_{\boldsymbol{\nu}}$
has an algebraic symplectic structure.
\end{Proposition}

We can obtain the above proposition by the following two propositions.

\begin{Proposition}\label{constructionof2-form} There is a non-degenerate relative $2$-form $\omega \in H^0 \big(\widehat{M}, \Omega^2_{\widehat{M}/N^{(n)}_r(e)}\big)$.
\end{Proposition}

\begin{proof}We set $\eta(s) := s - \iota\big(\tilde{\nabla}\big) \circ \mathrm{symb}_1(s) \in \mathcal{E}{\rm nd}\big(\tilde{E}\big) $, where $s \in \mathcal{F}^0$. For
\begin{gather*}
v=[(\{ u_{\alpha,\beta} \}, \{ (v_{\alpha}, w_{\alpha} )\})] \in {\mathbf H}^1\big(\mathcal{C} \times_T \widehat{U}, \mathcal{F}_{\widehat{U}}^{\bullet}\big) \qquad \text{and}\\
w=[(\{ u_{\alpha,\beta}' \}, \{ (v_{\alpha}',w_{\alpha}') \} )] \in {\mathbf H}^1\big(\mathcal{C} \times_T \widehat{U}, \mathcal{F}_{\widehat{U}}^{\bullet}\big),
\end{gather*} we put
\begin{gather}
\omega_1(v,w) = [ (\{ \operatorname{Tr}( \eta(u_{\alpha\beta}) \circ \eta(u_{\beta\gamma}')) \},
- \{ \operatorname{Tr} (\eta(u_{\alpha\beta}) \circ v_{\beta}') - \operatorname{Tr} (v_{\alpha} \circ \eta( u'_{\alpha\beta})) \} )]\qquad \text{and}\label{defof2form}\\
\omega_2(v,w) = [ \{ 2 \cdot \mathrm{symb}_1(u_{\beta\alpha}) \circ \mathrm{symb}_1(u'_{\beta\gamma}) \circ \psi_{\beta} \},{-}\{ \mathrm{symb}_1(u_{\beta\alpha}) \circ w_{\beta}'+ w_{\alpha}
\circ \mathrm{symb}_1(u'_{\alpha\beta}) \} ].\!\nonumber
\end{gather}
For each affine open subset $\widehat{U} \subset \widehat{M}$, we define a pairing
\begin{gather*}
{\mathbf H}^1\big(\mathcal{C} \times_T \widehat{U}, \mathcal{F}_{\widehat{U}}^{\bullet}\big) \otimes
{\mathbf H}^1\big(\mathcal{C} \times_T \widehat{U}, \mathcal{F}_{\widehat{U}}^{\bullet}\big)
\longrightarrow {\mathbf H}^2\big(\mathcal{C} \times_T \widehat{U}, \Omega_{\mathcal{C}\times_T \widehat{U}/\widehat{U}}^{\bullet}\big) \cong H^0\big(\mathcal{O}_{\widehat{U}}\big), \\
v\otimes w \longmapsto \omega_1(v,w) + \omega_2(v,w),
\end{gather*}
where we consider in \v{C}ech cohomology with respect to an affine open covering $\{ U_{\alpha} \}$ of $\mathcal{C} \times_TU$, $\{ u_{\alpha\beta} \} \in C^1\big(\mathcal{F}^0\big)$, $\{ (v_{\alpha}, w_{\alpha}) \} \in C^0\big(\mathcal{F}^1\big)$ and so on. This pairing determines a pairing
\begin{gather*}
\omega \colon \ \bold{R}^1 (\pi_{\widehat{M}})_*(\mathcal{F}^{\bullet}) \otimes \bold{R}^1 (\pi_{\widehat{M}})_*(\mathcal{F}^{\bullet}) \longrightarrow \mathcal{O}_{\widehat{M}}.
\end{gather*}
By the same argument as in the proof of \cite[Proposition~7.2]{Inaba}, $\omega$ is skew symmetric and non-degenerate.
\end{proof}

\begin{Proposition}\label{Prop closed form} For the $2$-form constructed in Proposition~{\rm \ref{constructionof2-form}}, we have $d \omega =0$.
\end{Proposition}

\begin{proof}Let $\Theta_{\widehat{M}_{\boldsymbol{\nu}}}^{\text{initial}}$ be the subbundle of $\Theta_{\widehat{M}_{\boldsymbol{\nu}}}$ consisted by the images of the tangent morphism $\Theta_{\widehat{M}_{\boldsymbol{\nu}}/T^*\tilde{M}_{g,n} } \rightarrow \Theta_{\widehat{M}_{\boldsymbol{\nu}}}$ and let $\Theta_{\widehat{M}_{\boldsymbol{\nu}}}^{\text{IMD}}$ be the subbundle of $\Theta_{\widehat{M}_{\boldsymbol{\nu}}}$ consisted by the images of $\widehat{\mathfrak{D}}\big(\hat{\pi}_{\boldsymbol{\nu}}^* \big(\Theta_{T^*\tilde{M}_{g,n}}\big)\big) \rightarrow \Theta_{\widehat{M}_{\boldsymbol{\nu}}}$. We take an affine open set $\widehat{U}\subset \widehat{M}_{\boldsymbol{\nu}}$. We have a canonical decomposition
\begin{gather*}
{\mathbf H}^1\big(\mathcal{F}_{\widehat{U}}^{\bullet}\big) \longrightarrow \Theta^{\text{initial}}_{\widehat{U}} \oplus \Theta^{\text{IMD}}_{\widehat{U}} , \qquad v=[\{ u_{\alpha\beta} \} , \{ (v_{\alpha},w_{\alpha}) \}] \longmapsto v_{\text{initial}}+ v_{\text{IMD}},
\end{gather*}
where
\begin{gather*}
v_{\text{initial}} = [\{ \eta (u_{\alpha\beta}) \} , \{ (v_{\alpha},0) \}] \qquad \text{and} \qquad v_{\text{IMD}} = \big[ \big\{ \iota\big(\tilde{\nabla}\big) \circ \mathrm{symb}_1(u_{\alpha\beta}) \big\}, \{(0 ,w_{\alpha} )\} \big].
\end{gather*}
We may assume that $\boldsymbol{\nu}$ is generic. Let $\widehat{U}$ be an affine open set of $\widehat{M}_{\boldsymbol{\nu}}$ and let $\big(\tilde{E},\tilde{\nabla} ,\big\{ \tilde{l}^{(i)}_j\big\}, \tilde{\psi}\big)$ be the family on $ \mathcal{C} \times_{\tilde{M}_{g,n}} \widehat{U}$. We take an affine open covering $\mathcal{C}_{\widehat{U}} = \bigcup_{\alpha} U_{\alpha}$ such that $\phi_{\alpha} \colon \tilde{E}|_{U_{\alpha}} \xrightarrow{\sim} \mathcal{O}^{\oplus r}_{U_{\alpha}}$ for any $\alpha$, $\sharp\{ i \,|\, \tilde{t}_i |_{\mathcal{C}_U} \cap U_{\alpha} \neq \varnothing \} \le 1$ for any $\alpha$ and $\sharp\{ \alpha \,|\, \tilde{t}_i |_{\mathcal{C}_U} \cap U_{\alpha} \neq \varnothing \} \le 1$ for any $i$. If we replace~$U_{\alpha}$ sufficiently smaller, there exists a sheaf $E_{\alpha}$ on $U_{\alpha}$ such that $E_{\alpha}|_{U_{\alpha}\cap U_{\beta}} \cong \big(\pi^{-1}_{\widehat{M}_{\boldsymbol{\nu}}} \mathcal{O}_{\widehat{M}_{\boldsymbol{\nu}}}|_{U_{\alpha}\cap U_{\beta}} \big)^{\oplus r^2}$ for any $\beta \neq \alpha$ and an isomorphism $\phi_{\alpha} \colon j_*\big(\widehat{\boldsymbol{V}}\big)|_{U_{\alpha}} \xrightarrow{\sim} E_{\alpha}$. Here the local system $\widehat{\boldsymbol{V}}$ is defined in Section~\ref{SS infinit defor}. For each $\alpha$, $\beta$, we put
\begin{gather*}
\varphi_{\alpha\beta} \colon \ E_{\beta}|_{U_{\alpha}\cap U_{\beta}} \xrightarrow{\phi_{\beta}^{-1}} j_*\big( \widehat{\boldsymbol{V}}\big)|_{U_{\alpha}\cap U_{\beta}} \xrightarrow{\phi_{\alpha}} E_{\alpha}|_{U_{\alpha}\cap U_{\beta}}.
\end{gather*}
For each $\alpha$, $\beta$, let $\mu_{\alpha\beta} \colon U_{\alpha\beta} \rightarrow U_{\alpha\beta}$ be an isomorphism such that the glueing scheme of the collection $(U_{\alpha}, U_{\alpha\beta}, \mu_{\alpha\beta})$ is isomorphic to $\mathcal{C}_{\widehat{U}}$.

We consider a vector field $v \in H^0 \big(\widehat{U},\Theta_{\widehat{U}}\big)$. Then $v$ corresponds to a derivation $D_{v} \colon \mathcal{O}_{\widehat{U}} \rightarrow \mathcal{O}_{\widehat{U}}$ which naturally induces a morphism
\begin{gather*}
D_{v} \colon \ \mathcal{H}{\rm om} (E_{\beta}|_{U_{\alpha}\cap U_{\beta}} , E_{\alpha}|_{U_{\alpha}\cap U_{\beta}}) \longrightarrow \mathcal{H}{\rm om} (E_{\beta}|_{U_{\alpha}\cap U_{\beta}} , E_{\alpha}|_{U_{\alpha}\cap U_{\beta}}).
\end{gather*}

The isomorphism $\Theta_{\widehat{M}_{\boldsymbol{\nu}}} \cong \bold{R}^1(\pi_{\widehat{M}_{\boldsymbol{\nu}}})_*\big(\big(\widehat{\mathcal{F}}^{\bullet}\big)^{{\rm an}}\big)$ is given by
\begin{gather*}
\Theta_{\widehat{M}_{\boldsymbol{\nu}}} \ni v \longmapsto
\big[\big\{ \big( \phi_{\alpha}^{-1} \circ D_v (\varphi_{\alpha\beta}) \circ \phi_{\beta}, D_v (\mu_{\alpha\beta}) \big) \big\}
, \{ D_v(\psi|_{U_{\alpha}}) \}\big]
\in \bold{R}^1(\pi_{\widehat{M}_{\boldsymbol{\nu}}})_*\big(\big(\widehat{\mathcal{F}}^{\bullet}\big)^{{\rm an}}\big),
\end{gather*}
and the 2-form $\omega(u,v)=\omega_1(u,v) +\omega_2(u,v)$, $u,v \in \Theta_{\widehat{M}}$,
is given by
\begin{gather*}
\omega_1(u,v) = [ \{ \operatorname{Tr} \left( D_{u_{\text{initial}}} (\varphi_{\alpha\beta})
 D_{v_{\text{initial}}} (\varphi_{\beta\gamma}) \varphi_{\gamma\alpha} \right) \}] \qquad \text{and}\\
\omega_2(u,v) =[ \{2 D_{u_{\text{IMD}}}(\mu_{\beta\alpha}) D_{v_{\text{IMD}}}(\mu_{\beta\alpha}) \tilde{\psi}_{\beta}|_{U_{\alpha\beta}} \}, \\
\hphantom{\omega_2(u,v) =[}{} \{ - D_{u_{\text{IMD}}}(\mu_{\beta\alpha})
D_{v_{\text{IMD}}}(\tilde{\psi}_{\beta}|_{U_{\alpha\beta}})
- D_{u_{\text{IMD}}}(\tilde{\psi}_{\alpha}|_{U_{\alpha\beta}}) D_{v_{\text{IMD}}}(\mu_{\alpha\beta}) \}].
\end{gather*}
Since the image of $\Theta^{\text{IMD}}_{\widehat{U}}$ under the tangent morphism of $\widehat{M}_{\boldsymbol{\nu}} \rightarrow M_{\boldsymbol{\nu}}$ determines the foliation determined by the isomonodromic deformations, we can show that
\begin{gather*}
d \omega_1 (u,v,w) = d \omega_1 (u_{\text{initial}},v_{\text{initial}},w_{\text{initial}}).
\end{gather*}
We have $d \omega_1 (u_{\text{initial}},v_{\text{initial}},w_{\text{initial}})=0$ by \cite[Proposition~7.3]{Inaba}. We can also show that $d \omega_2 (u,v,w) =0$. Then we have the closeness of $\omega=\omega_1 +\omega_2$.
\end{proof}

\subsection{Extended phase spaces of isomonodromic deformations}\label{Extended phase space}

\begin{Proposition}\label{Main Thm 1} The morphism $\hat{\pi}_{\boldsymbol{\nu}} \colon\widehat{M}_{\mathcal{C}/\tilde{M}_{g,n}}^{\boldsymbol{\alpha}}\big(\tilde{\boldsymbol{t}},r,e\big)_{\boldsymbol{\nu}} \rightarrow T^* \tilde{M}_{g,n}$ is a Poisson map.
\end{Proposition}
\begin{proof}

Let $\hat{\pi}_{\boldsymbol{\nu}}^t \colon \Theta_{\widehat{M}_{\boldsymbol{\nu}}} \rightarrow \pi^* \Theta_{T^*\tilde{M}_{g,n}}$ be the tangent morphism. We denote by $\tilde{\xi}\colon \Theta_{T^*\tilde{M}_{g,n}} \rightarrow \Omega^1_{T^*\tilde{M}_{g,n}}$ and $\xi \colon \Theta_{\widehat{M}_{\boldsymbol{\nu}}} \rightarrow \Omega^1_{\widehat{M}_{\boldsymbol{\nu}}}$ the homomorphisms induced by the symplectic structures on $T^*\tilde{M}_{g,n}$ and $\widehat{M}_{\boldsymbol{\nu}}$, respectively. The assertion follows from that the following diagram
\begin{gather}\label{diagramHatD}\begin{split}&
\xymatrix{\pi^* \Theta_{T^*\tilde{M}_{g,n}} \ar[d]^-{\widehat{\mathfrak{D}} }
\ar[r]^-{\pi^*(\tilde{\xi})}
& \pi^* \Omega^1_{T^*\tilde{M}_{g,n}} \ar[d] \\
\Theta_{\widehat{M}_{\boldsymbol{\nu}}} \ar[r]^-{\xi} & \Omega^1_{\widehat{M}_{\boldsymbol{\nu}}}}\end{split}
\end{gather}
is commutative and $\hat{\pi}_{\boldsymbol{\nu}}^t \circ \widehat{\mathfrak{D}} = \mathrm{id}$. Here, $\widehat{\mathfrak{D}} $ is the homomorphism (\ref{def hatD}).
\end{proof}

Let $\mu_1 , \ldots \mu_{3g-3 + n}$ be local vector fields on an affine open subset $U \subset \tilde{M}_{g,n}$. Let $h_i$ be a linear function on $T^* \tilde{M}_{g,n}$ corresponding to the local vector field $\mu_i$ on $U$. Assume that $\{ h_i,h_j \}_{\tilde{M}_{g,n}} = 0$ for $i,j=1 ,\ldots, 3g-3 + n$ and $dh_1 \wedge \cdots \wedge dh_{3g-3+n}$ is not identically $0$, where $\{\cdot , \cdot \}_{\tilde{M}_{g,n}}$ is the Poisson bracket associated to the symplectic structure $\omega_{\tilde{M}_{g,n}}$. Put $\widehat{U} = \big(\hat{\pi}_{\boldsymbol{\nu}}\circ p_{\tilde{M}_{g,n}}\big)^{-1}(U)$, where $\hat{\pi}_{\boldsymbol{\nu}} \colon \widehat{M}_{\mathcal{C}/\tilde{M}_{g,n}}^{\boldsymbol{\alpha}}\big(\tilde{\boldsymbol{t}},r,e\big)_{\boldsymbol{\nu}} \rightarrow T^* \tilde{M}_{g,n}$ and $p_{\tilde{M}_{g,n}} \colon T^*\tilde{M}_{g,n} \rightarrow \tilde{M}_{g,n}$. Let $\omega_{T^*\tilde{M}_{g,n}}$ be the symplectic structure on $T^*\tilde{M}_{g,n}$. We define a \textit{Hamiltonian} $E_i$ on $\widehat{U} $ as $ \hat{\pi}^* h_i$ for $i=1 ,\ldots, 3g-3 + n$. Let $a_i$ $(i=1,\ldots,3g-3+n)$ be constants. We call the \textit{Hamiltonian vector field on $\widehat{U}$ associated to $\sum\limits_{i=1}^{3g-3+n} a_i \mu_i$} the vector field $\sum\limits_{i=1}^{3g-3+n} a_i\{ \cdot, E_i \}$ on~$\widehat{U}$.

\begin{Proposition}\label{Prop autonomous Hamil system} First, the Hamiltonians $E_i$ satisfy $\{ E_i,E_j \} = 0$ for $i,j=1 ,\ldots, 3g-3 + n$. In particular, the functions $E_i$ are conserved quantities associated to the Hamiltonian vector fields. Second, the restriction of the Hamiltonian vector field associated to $\sum\limits_{i=1}^{3g-3+n} a_i \mu_i$ to the common level surface $E_1= 0, \dots , E_{3g-3+n} = 0$ in $\widehat{U}$, is coincide with $\mathfrak{D}\Big(\sum\limits_{i=1}^{3g-3+n} a_i \mu_i\Big)$, which is a vector field associated to isomonodromic deformations. Here we consider the vector field $\sum\limits_{i=1}^{3g-3+n} a_i \mu_i$ as an element of $\pi_{\boldsymbol{\nu}}^{*} (\Theta_{\tilde{M}_{g,n}}) \big(\widehat{U}\big)$.
\end{Proposition}

\begin{proof}Let $\{ \cdot , \cdot \}_{\tilde{M}_{g,n}}$ be the Poisson bracket associated to the symplectic structure $\omega_{T^*\tilde{M}_{g,n}}$ on~$T^*\tilde{M}_{g,n}$. Let $v_{h_i}$ be the element $\hat{\pi}_{\boldsymbol{\nu}}^{*} (\Theta_{T^*\tilde{M}_{g,n}}) \big(\widehat{U}\big)$ defined by the vector field $\{ \cdot, h_i\}_{\tilde{M}_{g,n}}$ on $\hat{\pi}\big(\widehat{U}\big) \subset \tilde{M}_{g,n}$. In other words, $\omega_{T^*\tilde{M}_{g,n}} (v_{h_i}, v ) = dh_i (v)$ for any $v \in \Theta_{U}$. Put $v_{h_i} = \big[\big\{ d^{h_i}_{\alpha\beta} \big\} , \big\{ w^{h_i}_{\alpha} \big\}\big] \in {\mathbf H}^1\big(\mathcal{G}_{\widehat{U}}^{\bullet}\big)$, where $\mathcal{G}^{\bullet}$ is the complex (\ref{complexG}). Put $v^{\text{IMD}}_{h_i} = \big[\big\{ \iota \big(\tilde{\nabla}\big)(d^{h_i}_{\alpha\beta} ) \big\},\big\{(0, w^{h_i}_{\alpha}) \big\}\big]\! $ $\in {\mathbf H}^1 \big(\mathcal{F}_{\widehat{U}}^{\bullet}\big)$. By the diagram (\ref{diagramHatD}), we have $\omega\big(v^{\text{IMD}}_{h_i}, v \big) = d E_i (v)$ for any $v \in \Theta_{\widehat{U}}$, that is, $v^{\text{IMD}}_{h_i}= \{ \cdot, E_i \}$, which is the Hamiltonian vector field associated to~$\mu_i$.

Note that $\{ E_i ,E_j\} = \omega\big(v^{\text{IMD}}_{h_i}, v^{\text{IMD}}_{h_j} \big) = \omega_{T^*\tilde{M}_{g,n}} (v_{h_i}, v_{h_j} ) = \{ h_i , h_j\}_{\tilde{M}_{g,n}}$. By the assumption that the linear functions $h_i$ satisfy $\{ h_i , h_j\}_{\tilde{M}_{g,n}}=0$, we have $\{ E_i ,E_j\} =0$. The common level surface $E_1= \cdots = E_{3g-3+n} = 0$ is $M_{\mathcal{C}/\tilde{M}_{g,n}}^{\boldsymbol{\alpha}}\big(\tilde{\boldsymbol{t}},r,e\big)_{\boldsymbol{\nu}}$. On this common level surface, the vector field associated to the Hamiltonian vector field of $\mu_i$ is $ \big[\big\{ \iota \big(\tilde{\nabla}\big)\big(d^{h_i}_{\alpha\beta} \big) \big\},\{0\}\big] \in {\mathbf H}^1 \big(\mathcal{F}^0_{\widehat{U}} \rightarrow \widetilde{\mathcal{F}}^1_{\widehat{U}} \big)$, which is a vector field associated to the isomonodromic deformations.
\end{proof}

\section[Moduli stack of stable parabolic connections with a quadratic differential and twisted cotangent bundle]{Moduli stack of stable parabolic connections\\ with a quadratic differential and twisted cotangent bundle}\label{Section Twisted Cotangent Bundle}

Let $C$ be a smooth projective curve of genus $g$, $g\ge 2$. The map from the moduli space of pairs $(E,\nabla)$ to the moduli space of vector bundles defined by $(E,\nabla) \mapsto E$ is a twisted cotangent bundle on the moduli space of vector bundles. Here, $E$ is a rank $r$ vector bundle on the fixed curve $C$ and $\nabla$ is a holomorphic connection on $E$. This twisted cotangent bundle has been investigated by Faltings, Ben-Zvi--Biswas, and Ben-Zvi--Frenkel (see \cite[Section~4]{Faltings}, \cite[Section~5]{BE}, \cite[Section~5]{BB1}, and \cite[Section~4.1]{BF}). Moreover, Ben-Zvi--Biswas and Ben-Zvi--Frenkel studied on a twisted cotangent bundle on the moduli space of pairs $(C,E)$ (see \cite[Section~6]{BB1} and \cite[Section~4.3]{BF}). In \cite{BB1,BB2}, Ben-Zvi and Biswas have introduced \textit{extended connections}, which are generalization of holomorphic connections.
We can define a natural map from the moduli space of extended connections to the moduli space of pairs $(C,E)$. This map is generalization of the map $(E, \nabla) \mapsto E$ and has been investigated in \cite[Section~6]{BB1} and \cite[Section~4.3]{BF}. In this section, we consider parabolic connections instead of holomorphic connections and study the moduli space of parabolic connections with a quadratic differential instead of the moduli space of extended connections. The purpose of this section is to show that the moduli space of parabolic connections with a quadratic differential is equipped with structure of a twisted cotangent bundles. In Section~\ref{2018.7.22.13.45}, we consider the moduli stack corresponding to the moduli scheme considered in the previous section. We introduce the moduli stack of pointed smooth projective curves and quasi-parabolic bundles.
We consider the cotangent bundle of this moduli stack. We describe the tangent sheaf of the total space of this cotangent bundle and the canonical symplectic form on this cotangent bundle. In Section~\ref{SS Moduli stack TCB}, we consider a map from the moduli stack of parabolic connections with a quadratic differential to the moduli stack of pointed smooth projective curves and quasi-parabolic bundles. We endow this map with structure of a~twisted cotangent bundle. In Section~\ref{EPC}, we introduce \textit{extended parabolic connections}, which are generalization of parabolic connections and also extended connections. We consider a relation between parabolic connections with a quadratic differential (which are also generalization of parabolic connections) and extended parabolic connections.

In this section, we assume that $\boldsymbol{\nu}$ is generic. If $\boldsymbol{\nu}$ is generic, then any $(\boldsymbol{t}, \boldsymbol{\nu})$-parabolic connection is irreducible. So all $(\boldsymbol{t}, \boldsymbol{\nu})$-parabolic connections are stable.

\subsection{Moduli stack of stable parabolic connections with a quadratic differential}\label{2018.7.22.13.45}

Let $\mathfrak{M}_{g,n}$ be the moduli stack of $n$-pointed smooth projective curves of genus $g$, where $n$-points consist of distinct points. Let $\widehat{\mathfrak{M}}_{g,n}(r,e, \boldsymbol{\nu})$ be the moduli stack of collections $((C,\boldsymbol{t}, \psi), (E, \nabla ,\boldsymbol{l}))$, where $(C, \boldsymbol{t} )$, $\boldsymbol{t}=(t_1, \ldots,t_n)$, is an $n$-pointed smooth projective curve of genus $g$ over $\mathbb{C}$ where $t_1, \ldots,t_n$ are distinct points, $\psi$ is an element of $H^0 \big(C, \Omega_C^{\otimes 2}(D(\boldsymbol{t}))\big)$, and $(E, \nabla ,\boldsymbol{l})$ is a $(\boldsymbol{t}, \boldsymbol{\nu})$-parabolic connection of rank~$r$ and of degree~$e$ on $C$. Let $\Theta_{\widehat{\mathfrak{M}}_{g,n}(r,e, \boldsymbol{\nu})}$ be the tangent complex of $\widehat{\mathfrak{M}}_{g,n}(r,e, \boldsymbol{\nu})$, that is, for each smooth map $f_U \colon U \rightarrow \widehat{\mathfrak{M}}_{g,n}(r,e, \boldsymbol{\nu})$ from a scheme $U$, the pull-back $f_U^{*} \Theta_{\widehat{\mathfrak{M}}_{g,n}(r,e, \boldsymbol{\nu})}$ is $\Theta_{U/\widehat{\mathfrak{M}}_{g,n}(r,e, \boldsymbol{\nu})} \rightarrow \Theta_{U}$ considered as a length $2$ complex supported in degree~$-1$ and~$0$. Here $\Theta_{U/\widehat{\mathfrak{M}}_{g,n}(r,e, \boldsymbol{\nu})} := \Delta^* \big(\Theta_{(U \times_{\widehat{\mathfrak{M}}_{g,n}(r,e, \boldsymbol{\nu})}U)/U} \big) $, where $U \rightarrow U \times_{\widehat{\mathfrak{M}}_{g,n}(r,e, \boldsymbol{\nu})}U$ is the diagonal. Let $\Theta_{\widehat{\mathfrak{M}}_{g,n}(r,e, \boldsymbol{\nu}), x}$ be the fiber of $\Theta_{\widehat{\mathfrak{M}}_{g,n}(r,e, \boldsymbol{\nu})}$ over a point $x = ((C,\boldsymbol{t}, \psi), (E, \nabla ,\boldsymbol{l}))$ of $\widehat{\mathfrak{M}}_{g,n}(r,e, \boldsymbol{\nu})$. Then $H^0(\Theta_{\widehat{\mathfrak{M}}_{g,n}(r,e,\boldsymbol{\nu}), x})$ is isomorphic to ${\mathbf H}^1\big(\mathcal{F}_x^{\bullet}\big)$. Here, we recall the complex~$\mathcal{F}_x^{\bullet}$:
\begin{gather*}
 \mathcal{F}_x^0 := \big\{ s \in \mathcal{A}_{E}(D(\boldsymbol{t}))\,|\, (s - \iota(\nabla) \circ \mathrm{symb}_1(s)) |_{t_i} \big(l^{(i)}_j\big) \subset l^{(i)}_j \text{ for any $i$, $j$} \big\},\\
\widetilde{\mathcal{F}}_x^1 := \big\{ s \in \mathcal{E}{\rm nd}(E) \otimes \Omega^1_{C} (D(\boldsymbol{t})) \,|\, {\sf res}_{t_i} (s) \big(l^{(i)}_j\big) \subset l^{(i)}_{j+1} \text{ for any $i$, $j$} \big\}, \\
 \mathcal{F}^1_x := \widetilde{\mathcal{F}}_x^1 \oplus \Omega_{C}^{\otimes 2} (D(\boldsymbol{t})); \text{ and } d_{\mathcal{F}^{\bullet}} := (d_{\nabla} ,d_{\psi} ) \circ \big(\mathrm{Id}
- \iota\big(\tilde{\nabla}\big) \circ \mathrm{symb}_1,\mathrm{symb}_1\big) \colon \mathcal{F}_x^0 \longrightarrow \mathcal{F}^1_x,
\end{gather*}
where $d_{\nabla} \colon \widetilde{\mathcal{F}}^0 \rightarrow \widetilde{\mathcal{F}}^1$, $s \mapsto \nabla \circ s - s \circ \nabla $ and $d_{\psi} \colon \Theta_{C} (-D(\boldsymbol{t}))\rightarrow \Omega_{C}^{\otimes 2} (D(\boldsymbol{t}))$ defined by~(\ref{dpsi Ui}). The pairing ${\mathbf H}^1\big(\mathcal{F}_x^{\bullet}\big) \otimes {\mathbf H}^1\big(\mathcal{F}_x^{\bullet}\big) \rightarrow {\mathbf H}^2\big(\Omega_C^{\bullet}\big)$ defined by~(\ref{defof2form}) gives a symplectic structure on $\widehat{\mathfrak{M}}_{g,n}(r,e, \boldsymbol{\nu})$.

\begin{Definition}Let $(C, \boldsymbol{t} )$ be an $n$-pointed smooth projective curve of genus $g$ over $\mathbb{C}$ where $t_1, \ldots,t_n$ are distinct points. We say $(E,\boldsymbol{l} )$, $\boldsymbol{l}=\big\{ l^{(i)}_* \big\}_{1\le i \le n}$, is a {\it quasi-parabolic bundle} of rank~$r$ and of degree~$e$ on $(C, \boldsymbol{t} )$ if $E$ is a rank $r$ algebraic vector bundle of degree~$e$ on $C$, and for each $t_i$, $l^{(i)}_* $ is a filtration $E|_{t_i} = l_0^{(i)} \supset l_1^{(i)} \supset \cdots \supset l_r^{(i)}=0$ such that $\dim \big(l_j^{(i)}/l_{j+1}^{(i)}\big)=1$, $j=0,1,\ldots,r-1$.
\end{Definition}
Let $\mathfrak{P}_{g,n}(r,e)$ be the moduli stack of pairs $ ((C,\boldsymbol{t}), (E, \boldsymbol{l}) ) $, where $(C, \boldsymbol{t} )$ ($\boldsymbol{t}=(t_1, \ldots,t_n)$) is an $n$-pointed smooth projective curve of genus $g$ over $\mathbb{C}$ where $t_1, \ldots,t_n$ are distinct points, and $(E, \boldsymbol{l})$ is a quasi-parabolic bundle of rank $r$ and of degree~$e$ on $(C, \boldsymbol{t} )$. We have a projection $\mathfrak{P}_{g,n}(r,e) \rightarrow \mathfrak{M}_{g,n}$. Let $\mathfrak{P}_{g,n} (r,e,\boldsymbol{\nu})$ be the substack defined by the condition where a quasi-parabolic bundle admits a $(\boldsymbol{t}, \boldsymbol{\nu})$-parabolic connection.
Let $\pi_{\mathfrak{P}_{g,n} (r,e,\boldsymbol{\nu})}$ and $\pi_{\mathfrak{M}_{g,n}}$ be the following morphisms:
\begin{gather*}
\pi_{\mathfrak{P}_{g,n} (r,e,\boldsymbol{\nu})} \colon \ \widehat{\mathfrak{M}}_{g,n}(r,e, \boldsymbol{\nu})
\longrightarrow \mathfrak{P}_{g,n} (r,e,\boldsymbol{\nu}), \qquad
 ((C,\boldsymbol{t}, \psi), (E, \nabla, \boldsymbol{l})) \longmapsto ((C,\boldsymbol{t}),( E, \boldsymbol{l})), \\
\pi_{\mathfrak{M}_{g,n}} \colon \ \mathfrak{P}_{g,n} (r,e,\boldsymbol{\nu}) \longrightarrow \mathfrak{M}_{g,n}, \qquad ((C,\boldsymbol{t}), ( E, \boldsymbol{l})) \longmapsto (C,\boldsymbol{t}) .
\end{gather*}
Let $\Theta_{\mathfrak{P}_{g,n} (r,e,\boldsymbol{\nu})}$ be the tangent complex of $\mathfrak{P}_{g,n} (r,e,\boldsymbol{\nu})$. Let $\Theta_{\mathfrak{P}_{g,n} (r,e,\boldsymbol{\nu}), p}$ be the fiber of $\Theta_{\mathfrak{P}_{g,n} (r,e,\boldsymbol{\nu})}\!\!$ over a point $ p = ((C,\boldsymbol{t}), (E, \boldsymbol{l}))$ of $\mathfrak{P}_{g,n} (r,e,\boldsymbol{\nu})$.

We consider infinitesimal deformations of $ p = ((C,\boldsymbol{t}), (E, \boldsymbol{l}))$. We put
\begin{gather*}
\widetilde{\mathcal{H}}_p^0 := \big\{ s \in \mathcal{E}{\rm nd} ( E ) \, |\, s |_{t_i} \big(l^{(i)}_j\big) \subset l^{(i)}_j \text{ for any $i$, $j$} \big\}\qquad \text{and} \\
 \widetilde{\mathcal{H}}_p^1 := \big\{ s \in \mathcal{E}{\rm nd} (E) \otimes \Omega^1_{C} (D(\boldsymbol{t}))\,|\, {\sf res}_{t_i} (s) \big(l^{(i)}_j\big) \subset l^{(i)}_{j+1} \text{ for any $i$, $j$}\big\}.
\end{gather*}
Note that $\big(\widetilde{\mathcal{H}}_p^0\big)^* \otimes \Omega^1_{C} \cong \widetilde{\mathcal{H}}_p^1$. Put
\begin{gather*}
\mathcal{H}_{p}^0:= \big\{ s \in \mathcal{A}_{E}(D(\boldsymbol{t})) \subset \mathcal{E}{\rm nd}_{\mathbb{C}} (E) \,|\, s |_{t_i} \big(l^{(i)}_j\big) \subset l^{(i)}_j \text{ for any $i$, $j$} \big\}
\end{gather*}
and $\mathcal{H}_{p}^1 := \big(\mathcal{H}_{p}^0\big)^* \otimes \Omega^1_{C}$. Then we have exact sequences
\begin{gather*}
\xymatrix{0\ar[r] & \widetilde{\mathcal{H}}_{p}^0 \ar[r] & \mathcal{H}_{p}^0 \ar[r]^-{\mathrm{symb}_1} & \Theta_{C} (-D(\boldsymbol{t}))
\ar[r] & 0} \qquad \text{and} \\
\xymatrix{0\ar[r] & \Omega^{\otimes 2}_{C} (D(\boldsymbol{t})) \ar[r]^-{q} & \mathcal{H}_{p}^1
\ar[r]^-{\kappa} & \widetilde{\mathcal{H}}_{p}^1 \ar[r] & 0.}
\end{gather*}
We take an affine open covering $\{ U_i \}$ of $C$ so that we can take a trivialization $\phi_i \colon E|_{U_i} \cong \mathcal{O}_{U_i}^{\oplus r}$ of $E$ on each $U_i$ and the restriction of $\mathcal{H}_{p}^1$ to $U_i$ is $\mathcal{O}_{U_i}$-isomorphic to the direct sum $\big(\widetilde{\mathcal{H}}_{p}^1\big)_{U_i} \oplus \Omega^{\otimes 2}_{C} (D(\boldsymbol{t}))_{U_i}$. We fix trivializations of~$E$. On $U_i \cap U_i$, the transformation $\big(\mathcal{H}_{p}^1\big)_{U_i} \rightarrow \big(\mathcal{H}_{p}^1\big)_{U_j}$ is given by
\begin{gather}
 \left(\Phi_i(f_i) d f_i , \phi_i(f_i) df_i \otimes df_i\right) \longmapsto (\Phi_j (f_i) df_i, \phi_j (f_i) df_i \otimes df_i) \nonumber\\
\qquad{} :=\left( (\theta_{ij}^{-1} \Phi_i(f_i) \theta_{ij}) d f_i, \phi_i (f_i) df_i \otimes df_i
+ \operatorname{Tr} \left( \theta_{ij}^{-1} \Phi_i (f_i) \frac{\partial \theta_{ij}}{\partial f_i} \right) d f_i \otimes d f_i \right),\label{H transf i to j}
\end{gather}
where $\theta_{ij} := \phi_i \circ \phi^{-1}_j \colon \mathcal{O}_{U_i \cap U_j}^{\oplus r} \rightarrow \mathcal{O}_{U_i \cap U_j}^{\oplus r}$ is a transition function of~$E$. Then $H^0\big(\Theta_{\mathfrak{P}_{g,n} (r,e,\boldsymbol{\nu}), p} \big)$ is isomorphic to $H^1 \big(\mathcal{H}_{p}^0\big) $, and $H^0 \big(\mathcal{H}_{p}^1\big) $ is the dual of $H^1 \big(\mathcal{H}_{p}^0\big) $. The vector space $H^0 \big(\mathcal{H}_{p}^1\big) $ is the space of 1-forms at~$p$.

Put $ p = ((C,\boldsymbol{t}), (E, \boldsymbol{l}))$. Let $\widehat{\Phi}_p$ be an element of $H^0 \big(\mathcal{H}^1_p\big)$, which is described by $(\Phi_p, \phi_p)$ locally, where $\Phi_p df \in \widetilde{\mathcal{H}}^1_p$ and $\phi_p df \otimes df \in \Omega_C^{\otimes 2} (D(\boldsymbol{t}))$. We consider infinitesimal deformations of $\big(p , \widehat{\Phi}_p\big)$. For $\widehat{\Phi}_p$, we define a complex $d^0\big(\widehat{\Phi}_p\big) \colon \mathcal{H}^0_{p} \rightarrow \mathcal{H}^1_{p}$ as follows. For each affine open set $U\subset C$, we define the image of $a_U \partial/\partial f_U + \eta_U \in \mathcal{H}^0_{p}(U)$ as
\begin{gather*}
\left( \Phi_p d f_U \circ \eta_U - \eta_U \circ \Phi_p d f_U - \frac{\partial (a_U \Phi_p)}{\partial f_U } d f_U ,\right. \\
\left.\qquad \operatorname{Tr}\left( \frac{\partial \eta_U}{\partial f_U } \Phi_p d f_U \otimes d f_U \right)
- a_U \frac{\partial \phi_p}{\partial f_U } d f_U \otimes d f_U - 2 \frac{\partial a_U}{\partial f_U }\phi_p d f_U \otimes d f_U \right) .
\end{gather*}
We can show that this homomorphism on each $U$ gives a homomorphism $d^0\big(\widehat{\Phi}_p\big) \colon \mathcal{H}^0_{p} \rightarrow \mathcal{H}^1_{p}$. We consider the first hypercohomology ${\mathbf H}^1 \big(\mathcal{H}^{\bullet}_{p}\big)$ of $d^0\big(\widehat{\Phi}_p\big) \colon \mathcal{H}^0_{p} \rightarrow \mathcal{H}^1_{p}$. By the \v{C}ech cohomology, an element of ${\mathbf H}^1 \big(\mathcal{H}^{\bullet}_{p}\big)$ is described by $\big[ \big\{ a_{ij} \partial / \partial f_i +\eta_{ij} \big\} , \big\{ (\hat{v}_i, \hat{w}_i) \big\} \big]$, where
\begin{gather*}
a_{jk} \frac{\partial f_i}{\partial f_j} -a_{ik} + a_{ij} = 0, \\
\left( \theta_{ji}^{-1} \eta_{jk} \theta_{ji} + a_{jk} \theta_{ji}^{-1} \frac{\partial \theta_{ji}}{\partial f_j} \right)- \eta_{ik} + \eta_{ij} =0, \qquad \text{and} \\
\left( \theta_{ji}^{-1} \hat{v}_j \theta_{ji} ,\hat{w}_j
+ \operatorname{Tr} \left( \theta_{ji}^{-1} \hat{v}_j \frac{\partial \theta_{ji}}{\partial f_j} d f_j \right) \right) - ( \hat{v}_i , \hat{w}_i) = d^0\big(\widehat{\Phi}_p\big) (\eta_{ij})
\end{gather*}
for some affine open covering $\{ U_i \}_i$ of $C$. Infinitesimal deformations of $\big(p , \widehat{\Phi}_p\big)$ are parametrized by ${\mathbf H}^1 \big(\mathcal{H}^{\bullet}_{p}\big)$. Then the fiber of the tangent sheaf of the moduli stack of pairs $\big( ((C,\boldsymbol{t}), (E, \boldsymbol{l})), \widehat{\Phi}\big)$ (where $\widehat{\Phi} \in H^0 \big(\mathcal{H}_{p}^1\big) $) at a point $\big(p, \widehat{\Phi}_p\big)$ is isomorphic to ${\mathbf H}^1 \big(\mathcal{H}^{\bullet}_{p}\big)$. Moreover, we define a paring ${\mathbf H}^1 \big(\mathcal{H}^{\bullet}_{p}\big) \otimes {\mathbf H}^1 \big(\mathcal{H}^{\bullet}_{p}\big) \rightarrow {\mathbf H}^2 \big(\Omega^{\bullet}_{C}\big)$ by
\begin{gather*}
\big[ \big( \big\{ a_{ij} \partial / \partial f_i +\eta_{ij} \big\} , \big\{ \big(\hat{v}_i, \hat{w}_i\big) \big\} \big) \big] \otimes
 \big[ \big( \big\{ a'_{ij} \partial / \partial f_i +\eta'_{ij} \big\} , \big\{ \big(\hat{v}'_i, \hat{w}'_i\big) \big\} \big) \big] \nonumber\\
\qquad{} \longmapsto
 \big[ \big( \big\{ \operatorname{Tr}\big(\eta_{ji} \big(a'_{jk}\Phi_j\big)\big) + \operatorname{Tr} \big( \big(a_{ji} \Phi_j\big) \eta'_{jk} \big)
 - 2 a_{ji} a'_{jk} \phi_{j} \big\}, \nonumber\\
 \qquad \qquad{} -\big\{ {-}\operatorname{Tr}\big( \eta_{ji}\hat{v}'_j \big) + \big( a_{ji} \hat{w}_j'\big)
 - \operatorname{Tr}\big( \hat{v}_i \eta'_{ij} \big) + \big( \hat{w}_i a'_{ij} \big) \big\} \big) \big].\label{sympform of moduli P}
\end{gather*}
\begin{Proposition}This pairing gives a symplectic structure on the moduli stack of pairs $\big( ((C,\boldsymbol{t}),$ $(E, \boldsymbol{l})), \widehat{\Phi}\big)$.
\end{Proposition}
\begin{proof} We define a 1-form $\theta_{\mathfrak{P}_{g,n} (r,e,\boldsymbol{\nu})}$ by
\begin{gather*}
\theta_{\mathfrak{P}_{g,n} (r,e,\boldsymbol{\nu})} \colon \ {\mathbf H}^1 \big(\mathcal{H}^{\bullet}_{p}\big) \longrightarrow H^1 \big( \Omega^1_{C} \big), \\
\big[ \big\{ a_{ij} \partial / \partial f_i +\eta_{ij} \big\} , \big\{ \big(\hat{v}_i, \hat{w}_i\big) \big\} \big]
 \longmapsto \left[ \operatorname{Tr} (\eta_{ij} \Phi_i df_i)+ a_{ji} \phi_i \left( \frac{\partial \mu_{ij}}{\partial f_{i}} df_{j} \right) \right]
\end{gather*}
for each $(\Phi_i df_i, \phi_i df_i \otimes df_i)$. This 1-form $\theta_{\mathfrak{P}_{g,n} (r,e,\boldsymbol{\nu})}$ is the canonical 1-form on the cotangent bundle of $\mathfrak{P}_{g,n} (r,e,\boldsymbol{\nu})$. Let $d \theta_{\mathfrak{P}_{g,n} (r,e,\boldsymbol{\nu})}$ be the exterior differential of $\theta_{\mathfrak{P}_{g,n} (r,e,\boldsymbol{\nu})}$. The 2-form $d \theta_{\mathfrak{P}_{g,n} (r,e,\boldsymbol{\nu})}$ gives the symplectic form on the cotangent bundle of $\mathfrak{P}_{g,n} (r,e,\boldsymbol{\nu})$. We compute the 2-form $d \theta_{\mathfrak{P}_{g,n} (r,e,\boldsymbol{\nu})}$ as follows
\begin{gather*}
 \operatorname{Tr} \big( D_{v'} (\theta_{ij}) D_{v} \big(\theta_{ij}^{-1} \Phi_i df_i\big)
- D_v (\theta_{ij}) D_{v'} \big(\theta_{ij}^{-1} \Phi_i df_i\big) \big) +d ( a_{ij} \operatorname{Tr} (\eta_{ij}' \Phi_i)) \\
\qquad{} =\operatorname{Tr} \big( {-}D_{v'} (\theta_{ij}) \theta_{ij}^{-1} D_{v} (\theta_{ij}) \theta_{ij}^{-1} \Phi_i df_i
+ D_v (\theta_{ij}) \theta_{ij}^{-1} D_{v'} (\theta_{ij}) \theta_{ij}^{-1} \Phi_i df_i \\
\qquad\quad{} +D_{v'} (\theta_{ij}) \theta_{ij}^{-1} D_{v} ( \Phi_i df_i)
-D_v (\theta_{ij}) \theta_{ij}^{-1} D_{v'} ( \Phi_i df_i) \big)+ d \big(a_{ij} \operatorname{Tr} ( \eta_{ij}' \Phi_i )\big) \\
\qquad{} =\operatorname{Tr}\big({-} \eta_{ij} \big( [ \Phi_i d f_i, \eta'_{ij} ] - d (a_{ij} \Phi_i) \big) - \eta_{ij} \hat{v}'_i
+ \eta'_{ij} \hat{v}_i \big) +a_{ij} \operatorname{Tr} (d ( \eta_{ij}') \Phi_i) \\
\qquad{} = - \operatorname{Tr} ( \eta_{ij} \hat{v}'_j) + \operatorname{Tr} (\eta_{ij}' \hat{v}_i) + a_{ij} \operatorname{Tr} ( d (\eta_{ij}') \Phi_i) \\
\qquad{} = \operatorname{Tr} ( \eta_{ji} \hat{v}'_j)+ \operatorname{Tr} ( \eta_{ij}' \hat{v}_i)
 - a_{ij} \operatorname{Tr} \left( \theta_{ji}^{-1} \hat{v}'_j \frac{\partial \theta_{ji}}{\partial f_i} \right)
 + a_{ij} \operatorname{Tr} ( d (\eta_{ij}') \Phi_i)
\end{gather*}
and
\begin{gather*}
 D_{v'} (\mu_{ji}) D_v\left( \phi_i \frac{\partial \mu_{ij}}{\partial f_j} \right) df_i- D_{v} (\mu_{ji}) D_{v'}\left( \phi_i \frac{\partial \mu_{ij}}{\partial f_j} \right) df_i + d\left( a'_{ji} a_{ji} \phi_j \right) \\
\qquad{} = a_{ij} \left( - a'_{ij} d \psi_i - 2 \phi_i \frac{\partial a'_{ij}}{\partial f_i} df_i \right)- a'_{ij} \hat{w}_i df_i + a_{ij} \hat{w}'_i df_i\\
\qquad{} = a_{ij} \hat{w}'_j df_j - a'_{ij} \hat{w}_i df_i + a_{ij} \operatorname{Tr} \left( \theta^{-1}_{ji} \hat{v}'_j \frac{\partial \theta_{ji}}{\partial f_i} \right)- a_{ij} \operatorname{Tr} ( d (\eta_{ij}') \Phi_i) .
\end{gather*}
Then we have this proposition.
\end{proof}

Put $ p = ((C,\boldsymbol{t}), (E, \boldsymbol{l}))$. Let $\nabla$ be a connection: $\nabla \colon E \rightarrow E \otimes \Omega^1_C (D(\boldsymbol{t}))$. For a connection $\nabla$, we define a decomposition of $H^0 \big(\mathcal{H}_{p}^1\big) $ as follows
\begin{gather*}
\begin{split} &
 H^0 \big(\mathcal{H}_{p}^1\big) \longrightarrow H^0 \big(\widetilde{\mathcal{H}}_{p}^1\big) \oplus H^0\big( \Omega^{\otimes 2}_{C} (D(\boldsymbol{t}))\big), \nonumber \\
& \widehat{\Phi} \longmapsto \big(\kappa\big(\widehat{\Phi}\big), \widehat{\Phi} - \psi\big(\nabla , \kappa\big(\widehat{\Phi}\big) \big) \big).
\end{split}
\end{gather*}
Here $\psi\big(\nabla , \kappa\big(\widehat{\Phi}\big)\big) \in H^0\big(\mathcal{H}_{p}^1\big)$ is defined as follows. We take an affine open covering $\{ U_i \}$ of $C$ such that on $U_i$ the connection $\nabla|_{U_i}$ is described by $d + A_i df_i$ and the Higgs field $\kappa\big(\widehat{\Phi}\big)|_{U_i}$ is described by $\Phi_i df_i$. On each $U_i$, we define an element $\psi\big(\nabla , \kappa\big(\widehat{\Phi}\big)\big) |_{U_i}$ as
\begin{gather*}
\psi\big(\nabla , \kappa\big(\widehat{\Phi}\big)\big) |_{U_i} = \left(\Phi _i df_i, \operatorname{Tr} \left( \Phi_i A_i + \frac{1}{2} \Phi_i \Phi_i \right) df_i \otimes df_i \right) \in \mathcal{H}_{p}^1(U_i) ,
\end{gather*}
which gives an element $\psi\big(\nabla , \kappa\big(\widehat{\Phi}\big)\big) \in H^0\big(\mathcal{H}_{p}^1\big)$.

\subsection{Moduli stack as twisted cotangent bundle}\label{SS Moduli stack TCB}

Let $\Gamma\big(\pi_{\mathfrak{P}_{g,n} (r,e,\boldsymbol{\nu})}\big)$ be the sheaf of set on $\mathfrak{P}_{g,n} (r,e,\boldsymbol{\nu})$ where $\Gamma\big(\pi_{\mathfrak{P}_{g,n} (r,e,\boldsymbol{\nu})}\big)(U)$ is the set of sections of $\pi_{\mathfrak{P}_{g,n} (r,e,\boldsymbol{\nu})}$ over $U$ for each smooth map $U\rightarrow \mathfrak{P}_{g,n} (r,e,\boldsymbol{\nu})$. Here $U$ is a scheme. We take a section $\sigma$ and put $\sigma(p) = (\nabla_p,\psi_p )$, where $\nabla_p \colon E \rightarrow E \otimes \Omega^1_C(D(\boldsymbol{t}))$ is a connection such that $(E, \boldsymbol{l}, \nabla_p)$ is a $(\boldsymbol{t}, \boldsymbol{\nu})$-parabolic connection of rank~$r$ and of degree~$e$ on $C$, and $\psi_p \in H^0 \big(C,\Omega_C^{\otimes 2}(D(\boldsymbol{t}))\big)$ for $p = ((C,\boldsymbol{t}), (E, \boldsymbol{l}))$.

\begin{Definition}\label{Omega action} For a 1-form $\widehat{\Phi}$ on $\mathfrak{P}_{g,n} (r,e,\boldsymbol{\nu})$, we define a translation by
\begin{gather}
 \Gamma\big( \pi_{\mathfrak{P}_{g,n} (r,e,\boldsymbol{\nu})} \big) \longrightarrow \Gamma\big( \pi_{\mathfrak{P}_{g,n} (r,e,\boldsymbol{\nu})} \big),\nonumber \\
\sigma(p)=(\nabla_p, \psi_p) \longmapsto t_{\widehat{\Phi}} ( \sigma)(p) := \big(\nabla_p + \kappa\big(\widehat{\Phi}_p\big) , \psi_p
+\big( \widehat{\Phi}_p - \psi\big(\nabla_p , \kappa\big(\widehat{\Phi}_p\big)\big) \big) \big).\label{1sr translation}
\end{gather}
\end{Definition}

By this translation (\ref{1sr translation}), we have an $\Omega^1_{\mathfrak{P}_{g,n} (r,e,\boldsymbol{\nu})}$-torsor structure on $\Gamma\big(\pi_{\mathfrak{P}_{g,n} (r,e,\boldsymbol{\nu})}\big)$.

\begin{Theorem}\label{Main Thm 2}Assume that $\boldsymbol{\nu}$ is generic. Let $\omega$ be the symplectic form on $\widehat{\mathfrak{M}}_{g,n}(r,e, \boldsymbol{\nu})$. We define a map $c \colon \Gamma\big(\pi_{\mathfrak{P}_{g,n} (r,e,\boldsymbol{\nu})}\big) \rightarrow \Omega_{\mathfrak{P}_{g,n}(r,e,\boldsymbol{\nu})}^{2{\rm cl}}$ by $c(\gamma) = \gamma^*(\omega)$ for $\gamma \in \Gamma\big(\pi_{\mathfrak{P}_{g,n} (r,e,\boldsymbol{\nu})}\big)$. Then for any $\widehat{\Phi} \in \Omega_{\mathfrak{P}_{g,n} (r,e,\boldsymbol{\nu})}^1$ we have $c( t_{\widehat{\Phi}}(\gamma) ) = d \big(\widehat{\Phi}\big) + c(\gamma)$. That is, $\big(\Gamma\big(\pi_{\mathfrak{P}_{g,n} (r,e,\boldsymbol{\nu})}\big), c\big)$ is an $\Omega_{\mathfrak{P}_{g,n} (r,e,\boldsymbol{\nu})}^{\ge 1}$-torsor.
\end{Theorem}
By this theorem and the argument as in Section~\ref{TDO on SAV}, the morphism $\pi_{\mathfrak{P}_{g,n} (r,e,\boldsymbol{\nu})} \colon \widehat{\mathfrak{M}}_{g,n}(r,e, \boldsymbol{\nu})$ $\rightarrow \mathfrak{P}_{g,n} (r,e,\boldsymbol{\nu})$ is equipped with structure of a twisted cotangent bundle.

\begin{proof} Put $ p = ((C,\boldsymbol{t}), (E, \boldsymbol{l}))$. Let $v$ be an element of $H^1 \big(\mathcal{H}^0_p\big)$. We take a section $\sigma$ and put $\sigma(p) = (\nabla_p,\psi_p )$, where $\nabla_p \colon E \rightarrow E \otimes \Omega^1_C(D(\boldsymbol{t}))$ is a connection and $\psi_p \in H^0 \big(C,\Omega_C^{\otimes 2}(D(\boldsymbol{t}))\big)$.

We take an affine open covering $\{ U_i \}$ of $C$ such that elements of $H^1 \big(\mathcal{H}^0_p\big)$ are described by the \v{C}ech cohomology: $v= [ \{ a_{ij} \partial / \partial f_i +\eta_{ij} \} ] $, where $a_{ij} \partial / \partial f_i +\eta_{ij}\in \mathcal{H}^0_p (U_i \cap U_j)$. Let $\nabla_p(v, \epsilon)$, $\widehat{\Phi}_p (v, \epsilon)$, and $\psi_p(v, \epsilon)$ be the infinitesimal deformations of $\nabla_p$, $\widehat{\Phi}_p$, and $\psi_p$ associated to $v$ over $\operatorname{Spec} \mathbb{C}[\epsilon]$, respectively, where $\epsilon^2=0$. We take local descriptions of the connection, the Higgs field, and the quadratic differentials on $U_i$ as follows. The connection $\nabla_p(v, \epsilon)$ and the Higgs field $\kappa\big(\widehat{\Phi}_p\big) (v, \epsilon)$ are described as $d+ A_i df_i + \epsilon v_i$ $\mathrm{mod}\ d\epsilon$ and $\Phi_i df_i + \epsilon \hat{v}_i$ $\mathrm{mod}\ d\epsilon$ on $U_i$, respectively. Moreover, on $U_i$ the quadratic differentials $\psi_p(v, \epsilon)$ and $\big( \widehat{\Phi}_p- \psi\big(\nabla_p , \kappa\big(\widehat{\Phi}_p\big)\big)\big)(v, \epsilon)$ are described by $\psi_{i} df_i \otimes df_i + \epsilon w_{i} df_i \otimes df_i$ and $\phi_{i} df_i \otimes df_i + \epsilon \hat{w}_{i} df_i \otimes df_i$ $\mathrm{mod}\ d\epsilon$ on $U_i$, respectively.

We decompose $\widehat{\Phi}_p = \widehat{\Phi}_1 + \widehat{\Phi}_2$, where $\widehat{\Phi}_1= \psi\big(\nabla_p , \widehat{\Phi}_p\big)$ and $\widehat{\Phi}_2= \widehat{\Phi}_p - \psi\big(\nabla_p , \widehat{\Phi}_p\big)$.
Then we can compute the exterior differentials of $\widehat{\Phi}_1$ and $\widehat{\Phi}_2$ as follows
\begin{gather*}
d\widehat{\Phi}_1 (v, v') = \big[ \big( \big\{ \operatorname{Tr}( \eta_{ji} (a'_{jk} \Phi_j )) + \operatorname{Tr} ( (a_{ji} \Phi_j ) \eta'_{jk} )
 - \operatorname{Tr} ( 2 a_{ji}a'_{jk} A_j \Phi_j + a_{ji} a'_{jk} \Phi_j \Phi_j ) \big\} \\
\hphantom{d\widehat{\Phi}_1 (v, v')=}{} -\big \{ {-}\operatorname{Tr}( \eta_{ji}\hat{v}'_j ) +
\operatorname{Tr} \big( (a_{ji} A_j ) \hat{v}_j' + ( a_{ji} \Phi_j ) v_j' + ( a_{ji} \Phi_j ) \hat{v}_j' \big) \\
\hphantom{d\widehat{\Phi}_1 (v, v')=}{} - \operatorname{Tr}( \hat{v}_i \eta'_{ij} ) + \operatorname{Tr} \big( ( a'_{ij} A_i ) \hat{v}_i + ( a'_{ij} \Phi_i ) v_i
+ ( a'_{ij} \Phi_i ) \hat{v}_i \big) \big\} \big) \big], \\
d\widehat{\Phi}_2 (v, v') = \big[ \big( \{ - 2 a_{ji} a'_{jk} \phi_{j} \} ,-\{ a_{ji}\hat{w}'_j + \hat{w}_i a'_{ij} \} \big) \big].
\end{gather*}

On the other hand, the symplectic form is computed as follows. Put $c(\gamma)_1 := \gamma^*(\omega_1)$ and $c(\gamma)_2 := \gamma^*(\omega_2)$, where $\omega_1$ and $\omega_2$ are defined by (\ref{defof2form}) in the proof of Proposition~\ref{constructionof2-form}. We have
\begin{gather*}
c( t_{\widehat{\Phi}}(\gamma) )_1(v,v') = \big[ \big(\big\{ \operatorname{Tr}\big( ( \eta_{ij}- a_{ij} A_i - a_{ij} \Phi_i )
 \circ ( \eta'_{jk}- a'_{jk} A_j - a'_{jk} \Phi_j ) \big) \big\}, \\
\hphantom{c( t_{\widehat{\Phi}}(\gamma) )_1(v,v') =}{}
-\big\{ \operatorname{Tr} \big(( \eta_{ij}- a_{ij} A_i - a_{ij} \Phi_i ) \circ (v'_j + \hat{v}'_j)\big)\\
\hphantom{c( t_{\widehat{\Phi}}(\gamma) )_1(v,v') =}{}
 - \operatorname{Tr} \big((v_i + \hat{v}_i) \circ ( \eta'_{ij}- a'_{ij} A_i - a'_{ij}\Phi_i ) \big) \big\} \big)\big] \\
\hphantom{c( t_{\widehat{\Phi}}(\gamma) )_1(v,v') }{}
= \big[ \big(\big\{ {-}\operatorname{Tr}\big( ( \eta_{ji}- a_{ji} A_j - a_{ji} \Phi_j ) \circ ( \eta'_{jk}- a'_{jk} A_j - a'_{jk} \Phi_j ) \big) \big\}, \\
\hphantom{c( t_{\widehat{\Phi}}(\gamma) )_1(v,v') =}{}
 -\big\{ {-} \operatorname{Tr} \big(( \eta_{ji}- a_{ji} A_j - a_{ji} \Phi_j ) \circ (v'_j + \hat{v}'_j)\big) \\
\hphantom{c( t_{\widehat{\Phi}}(\gamma) )_1(v,v') =}{}
 - \operatorname{Tr} \big((v_i + \hat{v}_i) \circ ( \eta'_{ij}- a'_{ij} A_i - a'_{ij}\Phi_i ) \big) \big\} \big)\big]
\end{gather*}
and
\begin{gather*}
c( \gamma )_1(v,v') = \big[ \big(\big\{ \operatorname{Tr}\big( ( \eta_{ij}- a_{ij} A_i ) \circ ( \eta'_{jk}-a'_{jk} A_j ) \big) \big\}, \\
\hphantom{c( \gamma )_1(v,v') =}{} -\big\{ \operatorname{Tr} \big(( \eta_{ij}-a_{ij} A_i ) \circ v'_j \big) - \operatorname{Tr} \big(v_i \circ ( \eta'_{ij}- a'_{ij}A_i ) \big) \big\} \big)\big] \\
\hphantom{c( \gamma )_1(v,v') }{}
= \big[ \big(\big\{ {-}\operatorname{Tr}\big( ( \eta_{ji}- a_{ji} A_j ) \circ ( \eta'_{jk}-a'_{jk} A_j ) \big) \big\}, \\
\hphantom{c( \gamma )_1(v,v') =}{}
 - \big\{ {-}\operatorname{Tr} \big(( \eta_{ji}-a_{ji} A_j ) \circ v'_j \big) - \operatorname{Tr} (v_i \circ \big( \eta'_{ij}- a'_{ij}A_i ) \big) \big\} \big)\big].
\end{gather*}
Then we obtain
\begin{gather*} c( t_{\tilde{\Phi}}(\gamma) )_1(v,v') - c( \gamma )_1(v,v') = d\widehat{\Phi}_1 (v, v').\end{gather*} Moreover, we also can show that
\begin{gather*} c( t_{\tilde{\Phi}}(\gamma) )_2(v,v') - c( \gamma )_2(v,v') = d\widehat{\Phi}_2 (v, v').\end{gather*} Then we obtain that $\big(\Gamma\big(\pi_{\mathfrak{P}_{g,n} (r,e,\boldsymbol{\nu})}\big), c\big)$ is an $\Omega_{\mathfrak{P}_{g,n} (r,e,\boldsymbol{\nu})}^{\ge 1}$-torsor.
\end{proof}

\subsection{Extended parabolic connections}\label{EPC}

Let $(C, \boldsymbol{t} )$, $\boldsymbol{t}= t_1+\cdots + t_n$, be an $n$-pointed smooth projective curve of genus $g$ over $\mathbb{C}$ where $t_1, \ldots,t_n$ are distinct points. Put $D(\boldsymbol{t}) = t_1 +\cdots + t_n$. We describe a description of $(\boldsymbol{t}, \boldsymbol{\nu})$-parabolic connection with a quadratic differential in terms of a ``integral kernel" on $C\times C$ as in~\cite{BB1} and~\cite{BB2}. Let $p_1 \colon C\times C \rightarrow C$ and $p_2 \colon C\times C \rightarrow C$ be the first and second projections, respectively. Put $\mathcal{O}_C(*D(\boldsymbol{t})):= \varinjlim_m \mathcal{O}_C(m D(\boldsymbol{t}))$, and $\Omega^1_C(*D(\boldsymbol{t})):= \Omega^1_{C} \otimes \mathcal{O}_C(*D(\boldsymbol{t}))$. Let $\mathcal{E}{\rm nd}^0(E)\subset \mathcal{E}{\rm nd}(E)$ be the subbundle of traceless endmorphisms of~$E$. We define sheaves $\mathcal{K}_{D(\boldsymbol{t})}(E)$ on $C\times C$ as $\mathcal{K}_{D(\boldsymbol{t})}(E):= p_1^*\big(E\otimes \Omega^1_C(*D(\boldsymbol{t}))\big) \otimes p_2^*\big(E^*\otimes \Omega^1_C\big) (2\Delta)$, where $\Delta \subset C\times C$ is the diagonal. We have a natural injective morphism $\Omega_C^{\otimes 2} (*D(\boldsymbol{t})) \otimes \mathcal{E}{\rm nd}^0(E) \rightarrow \mathcal{K}_{D(\boldsymbol{t})}(E)|_{3 \Delta}$. We define a sheaf $\mathcal{E}{\rm xConn}_{D(\boldsymbol{t})}(E)$ on $3\Delta$ by
\begin{gather*}
\mathcal{E}{\rm xConn}_{D(\boldsymbol{t})}(E) = \big\{ s \in \mathcal{K}_{D(\boldsymbol{t})}(E)|_{3 \Delta}/ (\Omega_C^{\otimes 2} (*D(\boldsymbol{t})) \otimes \mathcal{E}{\rm nd}^0(E))\,|\, s|_{\Delta } = \mathrm{Id}_E \big\}.
\end{gather*}
Note that we can consider $s|_{2\Delta}$ as a connection $s|_{2\Delta}\colon E \rightarrow E \otimes \Omega^1_C(*D(\boldsymbol{t}))$. We consider elements of $\mathcal{E}{\rm xConn}_{D(\boldsymbol{t})}(E)$ as pairs of connections and quadratic differentials on $C$ locally. Let $U_i$ and~$U_j$ be open sets of~$C$. Let $(A_i , a_i)$ be an elements of $\mathcal{E}{\rm xConn}_{D(\boldsymbol{t})}(E)$ on~$U_i$. Here $A_i df_i$ is a~connection matrix on~$U_i$ and $a_i df_i \otimes df_i \in H^0\big(U_i, \Omega_C^{\otimes 2} (D(\boldsymbol{t}))\big)$. The transformation of the pair is the following
\begin{gather*}
(A_i d f_i, a_i df_i \otimes df_i) \\
\quad{} \longmapsto \left( \theta_{ij}^{-1}A_i \theta_{ij} df_i+ \theta_{ij}^{-1} \frac{d \theta_{ij}}{d f_i} d f_i,
 \left( a_i + \operatorname{Tr} \left( \theta_{ij}^{-1} A_i \frac{d \theta_{ij}}{d f_i} \right)
 +\frac{1}{2} \operatorname{Tr} \left(\theta_{ij}^{-1} \frac{d^2 \theta_{ij}}{d f_i^2} \right) \right) df_i \otimes df_i \right)
\end{gather*}
on $U_i \cap U_j$. We may define an $\mathcal{H}_{(C,\boldsymbol{t},E,\boldsymbol{l})}^1$-action on $\mathcal{E}{\rm xConn}_{D(\boldsymbol{t})}(E)$ for any parabolic structu\-res~$\boldsymbol{l}$ by~(\ref{H transf i to j}).

Let $\nabla \colon E\rightarrow E\otimes\Omega_C^1(D(\boldsymbol{t}))$ be a connection. We can define a global section $\nabla_{\text{Ex}}$ of $\mathcal{E}{\rm xConn}_{D(\boldsymbol{t})}(E)$ associated to $\nabla$ as follows. Take a trivialization of the locally free sheaf $E$ on an open set $U_{i}$ of $C$. Let $A_i df_i$ be the connection matrix of $\nabla$ on $U_{i}$. We define $\nabla_{\text{Ex}}|_{U_{i}} \in \mathcal{E}{\rm xConn}_{D}(E)(U_{i})$ as
\begin{gather*}
\left( A_i d f_i , \frac{1}{2} \operatorname{Tr} \left( d(A_i ) \otimes d f_{i} \right) + \frac{1}{2} \operatorname{Tr}\left(A_i d f_{i} \otimes A_i d f_{i}\right) \right),
\end{gather*}
where $d$ is the exterior derivative. Let $\big(\big(E, \nabla , \big\{ l_j^{(i)} \big\}\big), \psi\big)$ be a $(\boldsymbol{t}, \boldsymbol{\nu})$-parabolic connection with a~quadratic differential. For $(\nabla, \psi)$, we can define a global section $\nabla_{\text{Ex}}+\psi$ of $\mathcal{E}{\rm xConn}_{D(\boldsymbol{t})}(E)$ by the above construction. We call this description $\big(E, \nabla_{\text{Ex}} + \psi , \big\{ l_j^{(i)} \big\} \big)$ a~$(\boldsymbol{t}, \boldsymbol{\nu})$-\textit{extended parabolic connection}.

By the identification a parabolic connection with a quadratic differential as an extended parabolic connection, we obtain another $\Omega^1_{\mathfrak{P}_{g,n} (r,e,\boldsymbol{\nu})}$-torsor structure on $\Gamma\big(\pi_{\mathfrak{P}_{g,n} (r,e,\boldsymbol{\nu})}\big)$. We take a section $\sigma$ of $\pi_{\mathfrak{P}_{g,n} (r,e,\boldsymbol{\nu})}$ and put $\sigma(p) = (\nabla_p,\psi_p )$, where $\nabla_p \colon E \rightarrow E \otimes \Omega^1_C(D(\boldsymbol{t}))$ is a connection and $\psi_p \in H^0 \big(C,\Omega_C^{\otimes 2}(D(\boldsymbol{t}))\big)$ for $p = ((C,\boldsymbol{t}), (E, \boldsymbol{l}))$. For a connection $\nabla_p \colon E \rightarrow E \otimes \Omega^1_{C} (D(\boldsymbol{t}))$ and $\Phi_p \in H^0\big(\widetilde{\mathcal{H}}_{p}^1\big)$, we define $\psi'(\nabla_p , \Phi_p) \in H^0\big(\mathcal{H}_{p}^1\big)$ as follows. We take an affine open covering~$\{ U_i \}$ of $C$ such that on $U_i$ the connection $\nabla_p|_{U_i}$ is described by $d + A_i df_i$ and the Higgs field~$\Phi_p|_{U_i}$ is described by $\Phi_i df_i$. On each $U_i$, we define an element $\psi'(\nabla_p , \Phi_p) |_{U_i}$ as
\begin{gather*}
\psi'(\nabla_p , \Phi_p) |_{U_i} = \left( \Phi _i df_i, \operatorname{Tr} \left( \Phi_i A_i df_i \otimes df_i + \frac{1}{2} \Phi_i \Phi_i df_i \otimes df_i
+ \frac{1}{2} d (\Phi_i)\otimes df_i \right) \right) \in \mathcal{F}_{P_{\boldsymbol{\nu}}}^1(U_i) ,
\end{gather*}
which gives an element $\psi'(\nabla_p , \Phi_p) \in H^0\big( \mathcal{H}_{p}^1 \big)$. For a 1-form $\widehat{\Phi}$ on $\mathfrak{P}_{g,n} (r,e,\boldsymbol{\nu})$, we define a~translation by
\begin{gather*}
t'_{\widehat{\Phi}} \colon \ \Gamma\big(\pi_{\mathfrak{P}_{g,n} (r,e,\boldsymbol{\nu})}\big)\longrightarrow \Gamma\big(\pi_{\mathfrak{P}_{g,n} (r,e,\boldsymbol{\nu})}\big), \\
\hphantom{t'_{\widehat{\Phi}} \colon}{} \ \sigma(p) = (\nabla_p, \psi_p) \longmapsto t'_{\widehat{\Phi}}(\sigma ) (p) := \big( \nabla_p +\kappa \big(\widehat{\Phi}_p\big) , \psi_p -\big( \widehat{\Phi}_p- \psi'\big(\nabla_p , \kappa \big(\widehat{\Phi}_p\big) \big)\big)\big).
\end{gather*}
By the translation, we have another $\Omega^1_{\mathfrak{P}_{g,n} (r,e,\boldsymbol{\nu})}$-torsor structure on $\Gamma\big(\pi_{\mathfrak{P}_{g,n} (r,e,\boldsymbol{\nu})}\big)$.

\subsection*{Acknowledgements}
The author is supported by Grant-in-Aid for JSPS Research Fellows Number 18J00245. He is grateful to the anonymous referees' suggestions which helped to improve the paper.

\pdfbookmark[1]{References}{ref}
\LastPageEnding

\end{document}